\documentclass[uplatex,a4paper,reqno]{amsart}

\usepackage{tikz-cd}
\usepackage{amsmath}
\usepackage{amssymb}
\usepackage{enumerate}
\usepackage{ascmac}
\usepackage{mathrsfs}
\usepackage{multirow}
\usepackage{amsthm}
\usepackage{wrapfig}
\usepackage{here}
\usepackage{multirow}
\usepackage{hyperref}
\usepackage{graphicx,color}

\makeatletter
\@namedef{subjclassname@2020}{%
  \textup{2020} Mathematics Subject Classification}
\makeatother

\theoremstyle{plain}
    \newtheorem{thm}{Theorem}[section]
    \newtheorem{lem}[thm]{Lemma}
    \newtheorem{prop}[thm]{Proposition}
    \newtheorem{cor}[thm]{Corollary}

\theoremstyle{definition}    
    \newtheorem{defn}[thm]{Definition}
    
    \newtheorem{rem}[thm]{Remark}

\def\A{{\mathcal{A}}}
\def\an{{\mathrm{an}}}
\def\aru{{\arrow[u,"\rotatebox{90}{$\in$}",phantom]}}
\def\Aut{{\mathrm{Aut}}}

\def\C{{\mathbb{C}}}
\def\Cc{{\mathcal{C}}}
\def\can{{\mathrm{can}}}

\def\CH{{\mathrm{CH}}}
\def\D{{\mathcal{D}}}
\def\Dc{{\mathscr{D}}}

\def\del{\partial}

\def\div{{\mathrm{div}}}
\def\E{{\mathcal{E}}}
\def\ev{{\mathrm{ev}}}

\def\Frac{{\mathrm{Frac}}}

\def\id{{\mathrm{id}}}
\def\ind{{\mathrm{ind}}}
\def\Ker{{\mathrm{Ker}}}
\def\Km{{\mathrm{Km}}}
\def\L{{\mathscr{L}}}
\def\Lc{{\mathcal{L}}}

\def\lra{\:{\longrightarrow}\:}

\def\O{{\mathcal{O}}}

\def\P{{\mathbb{P}}}
\def\Pc{{\mathcal{P}}}
\def\Q{{\mathbb{Q}}}
\def\Qc{{\mathcal{Q}}}

\def\R{{\mathbb{R}}}

\def\rank{{\mathrm{rank}\:}}

\def\sgn{{\mathrm{sgn}}}
\def\Spec{{\mathrm{Spec}\,}}

\def\tr{{\mathrm{tr}}}
\def\X{{\mathcal{X}}}
\def\Y{{\mathcal{Y}}}
\def\Z{{\mathbb{Z}}}

\begin{document}
\title{Higher Chow cycles on a family of Kummer surfaces}
\author{Ken Sato}
\address{Department of Mathematics, Tokyo Institute of Technology}
\email{sato.k.da@m.titech.ac.jp}
\subjclass[2020]{Primary 14C15; Secondary 14J28}
\begin{abstract}
We construct a collection of families of higher Chow cycles of type $(2,1)$ on a 2-dimensional family of Kummer surfaces,
and prove that for a very general member, they generate a subgroup of rank $\ge 18$ in the indecomposable part of the higher Chow group.
Construction of the cycles uses a finite group action on the family, and 
the proof of their linear independence uses Picard-Fuchs differential operators.
\end{abstract}
\maketitle

\section{Introduction}
In the celebrated paper \cite{bloch}, Bloch defined the \textit{higher Chow groups} $\CH^p(X,q)$ for a variety $X$, which are generalizations of the classical Chow groups. 
They are identified with the motivic cohomology when $X$ is smooth.
Higher Chow groups appear in many aspects of algebraic geometry and number theory and are related to many important problems.
However, their structures are still mysterious for many varieties when the codimension $p$ is greater than 1. 
In general, it is not easy to construct non-trivial higher Chow cycles (cf. \cite{Fla}, \cite{Mil}, \cite{Spi}).

We study $\CH^2(X,1)$ for a certain type of Kummer surfaces $X$. 
A cycle $\xi\in \CH^2(X,1)$ is called \textit{indecomposable} if it is not contained in the image of the map
\begin{equation*}
\C^\times \otimes_{\Z}\mathrm{Pic}(X)=\CH^1(X,1)\otimes_\Z \CH^1(X) \lra \CH^2(X,1)
\end{equation*}
induced by the intersection product.
The image of this map is called the \textit{decomposable part}: this is an easily accessible part.
The cokernel of this map is called the \textit{indecomposable part} and is denoted by 
$\CH^2(X,1)_\ind$.
We are interested in the size of $\CH^2(X,1)_\ind$.
Explicit constructions of non-trivial elements of $\CH^2(X,1)_\ind$ for $K3$ surfaces $X$ were initiated by M\"uller-Stach \cite{MS}.
Since then, further examples have been constructed (\cite{Col}, \cite{dAMS}, \cite{CL}, \cite{Ker}, \cite{CDKL16}, \cite{Sas}). 

In this paper, we study a 2-dimensional family of Kummer surfaces $\widetilde{\X}\rightarrow T$ associated with products of elliptic curves. 
We construct a collection of families of higher Chow cycles on $\widetilde{\X}\rightarrow T$. 
They generate a subgroup $\Xi_t$ of $\CH^2(\widetilde{\X}_t,1)_{\ind}$ for every $t\in T$. 
Our main result is the following.
\begin{thm}$\mathrm{(Theorem\: \ref{mainthm})}$ For a very general\footnote{We use the word ``very general" for the meaning that ``outside of a countable union of proper(= not the whole space) analytic subsets".} $t\in T$, we have $\rank \Xi_t \ge 18$.
In particular, $\rank \CH^2(\X_t,1)_\ind \ge 18$.
\end{thm}
Our construction uses the realization of the Kummer surfaces as the desingularized double coverings of $\P^1\times \P^1$, and our construction of higher Chow cycles uses special types of $(1,1)$-curves on $\P^1\times \P^1$. 
This is similar to the construction of higher Chow cycles on abelian surfaces in \cite{Sre14}.
The fact that $\CH^2(\widetilde{\X}_t,1)_\ind \neq 0$ already follows from the result of Section 6 of \cite{CDKL16} because our family is a base change of the family considered there.
Our novelty here is the rank estimate $\rank \CH^2(\widetilde{\X}_t,1)_\ind\ge 18$.
To produce many indecomposable cycles, we use a group action on the family as we will explain below.

A standard way to detect indecomposability is to use the \textit{transcendental regulator map}
\begin{equation*}
\begin{tikzcd}
r: \CH^2(\widetilde{\X}_t,1) \arrow[r] &H^3_{\mathcal D}(\widetilde{\X}_t, \Z(2))\simeq  \dfrac{F^1H^2(\widetilde{\X}_t,\C)^\vee}{H_2(\widetilde{\X}_t,\Z)} \arrow[r,twoheadrightarrow] &\dfrac{H^{2,0}(\widetilde{\X}_t)^\vee}{H_2(\widetilde{\X}_t,\Z)}
\end{tikzcd}
\end{equation*}
where the first map is the Beilinson regulator map and the last is the natural projection. 
It is known that the map $r$ factors through $\CH^2(\widetilde{\X}_t,1)_\ind$. 
For a family $\{\xi_t\}_{t\in T}$ of higher Chow cycles on $\widetilde{\X}\rightarrow T$, 
their transcendental regulators give rise to a certain type of normal function $\nu_\tr(\xi)=\{r(\xi_t)\}_{t\in T}$.
If we take the pairing of $\nu_{\tr}(\xi)$ with a relative 2-form $\omega$ on $\widetilde{\X}\rightarrow T$, this gives a multivalued holomorphic function on $T$.

Our construction begins with an initial cycle family $\xi_1-\xi_0$ (see Section 4 for the definition).
For this cycle family, the above multivalued function is given by the improper integral
\begin{equation}\label{ldayo}
\Lc (a,b) = 2\int_{\triangle}\frac{dxdy}{\sqrt{x(1-x)(1-ax)}\sqrt{y(1-y)(1-by)}}
\end{equation}
where $\triangle =\{(x,y)\in \R^2\::\: 0<y<x<1\}$.
The integral (\ref{ldayo}) is similar to the integral representation of Appell's hypergeometric functions (though the boundary $\del\triangle$ is not necessarily contained in the branching locus of the integrand). 
What is important to us is that $\Lc$ satisfies the following system of inhomogeneous differential equations:
\begin{equation}\label{diffsystem}
\left\{
\begin{aligned}
a(1-a) \frac{\del^2\Lc}{\del a^2} +(1-2a)\frac{\del\Lc}{\del a} -\frac{1}{4}\Lc &= \frac{2}{a-b}\left(\frac{\sqrt{1-b}}{\sqrt{1-a}}-1\right) \\
b(1-b) \frac{\del^2\Lc}{\del b^2} +(1-2b)\frac{\del\Lc}{\del b} -\frac{1}{4}\Lc &= \frac{2}{a-b}\left(1-\frac{\sqrt{1-a}}{\sqrt{1-b}}\right)
\end{aligned}
\right.
\end{equation}
We denote the differential operators appearing in the left-hand side of (\ref{diffsystem}) by $\Dc_1$ and $\Dc_2$, respectively. 
Note that they are the Gauss hypergeometric differential operators of type $\left(\frac{1}{2},\frac{1}{2},1\right)$.
Since the periods of $\widetilde{\X}\rightarrow T$ are annihilated by $\Dc=(\Dc_1,\Dc_2)$, $\Dc$ is called a \textit{Picard-Fuchs differential operator} on $\widetilde{\X}\rightarrow T$. 
From the differential equations (\ref{diffsystem}), we see that the image of $(\xi_1)_t-(\xi_0)_t$ under the transcendental regulator is non-trivial for very general $t$.

The second part of our construction is to produce more cycles from $\xi_1-\xi_0$ by using a group action on $\widetilde{\X}\rightarrow T$. 
For our purpose, automorphisms which act trivially on $T$ are not sufficient.
We consider a finite group $G_{\X}$ (isomorphic to a $\Z/2$-extension of $(\mathfrak{S}_4\times_{\mathfrak{S}_3} \mathfrak{S}_4)^2$) which acts on the total space $\widetilde{\X}$ and also on the base $T$. 
By letting $G_{\X}$ act on $\xi_1-\xi_0$, we obtain the subgroup $\Xi_t$ of $\CH^2(\widetilde{\X}_t,1)_{\ind}$. 
To show that $\Xi_t$ has rank $\ge 18$, 
we consider the $G_{\X}$-action on the normal function of $\xi_1-\xi_0$ and the differential equations (\ref{diffsystem}) it satisfies. 
This is the second main calculation in this paper.

Inhomogeneous Picard-Fuchs differential equations arising from normal functions are studied recently (e.g. \cite{dAMS08}, \cite{Ker} and \cite{CDKL16}). 
The differential equations (\ref{diffsystem}) is a $\Q(0)$-extension of the exterior tensor product of the two Gauss hypergeometric differential equations.
Hence this is  an example where the normal function can be expressed by a variant of hypergeometric functions.
In \cite{AO}, Asakura and Otsubo constructed a family of varieties and cycles whose normal functions are expressed by the generalized hypergeometric function ${}_3F_2$. 
Our cycles give an analogous example.

\subsection*{Outline of the paper}
This paper is divided into two parts. Part 1 is devoted to the construction of the initial cycle family $\xi_1-\xi_0$ and the computation of its normal function. 
In Section 2, we recall basic facts about higher Chow cycles and the regulator map. 
In Section 3 and Section 4, we define the family of Kummer surfaces and higher Chow cycles. 
In Section 5, we compute the normal function of $\xi_1-\xi_0$ and deduce the indecomposability. 
Part 2 is devoted to the construction and calculation of the group action. 
In Section 6, we construct the finite group action on the Kummer family.
In Section 7, we construct more higher Chow cycles using the group action and calculate the images of their normal functions under the Picard-Fuchs differential operator.
This calculation is based on the transformation formula proved in Section 8.

\subsection*{List of notations}
We summarize the frequently used notations in this paper. 
Usually, $\widetilde{(-)}$ indicates a blow-up.
The subscript $0$ (e.g. $T_0$) means an initial space;
when it is taken out (e.g. $T\rightarrow T_0$), this means taking an \'etale cover or an \'etale base change (e.g. $\widetilde{\X}\rightarrow \widetilde{\X}_0$).

\begin{table}[H]
\begin{tabular}{c|c}
notation & explanation\\ \hline
$S_0$ & $\P^1$ minus $0,1,\infty$ \\
$T_0$ & an affine open subset of $S_0\times S_0$ \\
$T$ & a finite \'etale cover of $T_0$ \\
$A_0$ & the coordinate ring of $S_0$ \\
$B_0$ & the coordinate ring of $T_0$ \\
$B$  & the coordinate ring of $T$ \\
$\Y_0$ (resp. $\Y$) & a family of $\P^1\times \P^1$ over $T_0$ (resp. $T$)\\
$\X_0$ (resp. $\X$) & a singular double cover of $\Y_0$ (resp. $\Y$) \\
$\A_0$ (resp. $\A$) & a family of products of elliptic curves over $T_0$ (resp. $T$) \\
$\Sigma$ & a set of 4 sections of $\P^1\times S_0$ and $\E$ over $S_0$\\
$\Sigma^2$ & a set of 16 sections of $\Y, \X$ and $\A$ over $T$\\
$\widetilde{\X}_0$ (resp. $\widetilde{\X}$) & the family of Kummer surfaces over $T_0$ (resp. $T$)\\
$\widetilde{\A}_0$ (resp.  $\A$)  & the blowing-up of $\A_0$ (resp. $\A$) along $\Sigma^2$ \\
$\widetilde{\Y}_0$ (resp. $\widetilde{\Y}$) & the blowing-up of $\Y_0$ (resp. $\Y$) along $\Sigma^2$\\
$U$ & an affine open subset of $\Y$ \\
$V$, $W$  & affine open subsets of $\X$ \\
$\D$  & the family of diagonal lines on $\P^1\times \P^1$ over $T$ \\
$\Cc$  & the pull-back of $\D$ by $\X\rightarrow \Y$ \\
$\widetilde{\Cc}$ & the strict transform of $\Cc$ \\
$f(x),g(y)$ & $f(x)=x(1-x)(1-ax), g(y)=y(1-y)(1-by)\in \O_{\Y}(U)$ \\
$G_{T_0}$  & a group ($\simeq \mathfrak{S}_3\times \mathfrak{S}_3$) acting on $T_0$\\
$G_{T}$  & a group ($\simeq \mathfrak{S}_4\times \mathfrak{S}_4$) acting on $T$\\
$G_{\Y_0}$ & a group ($\simeq \mathfrak{S}_4\times \mathfrak{S}_4$) acting on $\Y_0$\\
$G_{\Y}$ & a group ($\simeq (\mathfrak{S}_4\times_{\mathfrak{S}_3} \mathfrak{S}_4)^2$) acting on $\Y$ \\
$G_{\X}$ & a group ($\simeq (\mathfrak{S}_4\times_{\mathfrak{S}_3} \mathfrak{S}_4)^2\times_{\mu_2} \mu_4$) acting on $\X$ and $\widetilde{\X}$ \\
$\eta(=\eta_1\eta_2)$ & a 1-cocycle of $G_{\Y_0}$ \\
$\phi_1$, $\phi_2$  & a 1-cocycle of $G_{T}$ \\
$\chi$  & a 1-cocycle of $G_{\X}$ \\
$\Psi_{\rho}$ & a $G_{\X}$-linearization of $\O_T^{\an}$ and $\Qc$ \\
$\Theta_{\rho}$ & a $G_{\X}$-linearization of $(\O_T^{\an})^{\oplus 2}$ \\
\end{tabular}
\end{table}

\subsection*{Acknowledgement}
The author expresses his sincere gratitude to his supervisor Professor Tomohide Terasoma, who lead the author to this subject. 
The ideas of construction of higher Chow cycles were taught by him.
Furthermore, he read the draft of this paper carefully and gave many valuable comments which simplifies the arguments in this paper. He also thanks Professor Shuji Saito and Professor Takeshi Saito sincerely, who gave the author many helpful comments on this paper. 
The author is very grateful to the referee for a lot of helpful advices for improving the exposition of the paper.
The author is supported by the FMSP program by the University of Tokyo.

\part{Construction of the initial higher Chow cycles}
\section{Preliminaries}
\subsection{Higher Chow cycles of type (2,1)}
For a smooth variety $X$ over $\C$, let $\CH^p(X,q)$ be the higher Chow group defined by Bloch.
In this paper, we treat the case $(p,q)=(2,1)$.
In this case, the following fact is well-known (see, e.g., \cite{MS2} Corollary 5.3).
\begin{prop}\label{gerstenprop}
The higher Chow group $\mathrm{CH}^2(X,1)$ is isomorphic to the middle homology group of the following complex.
\begin{equation*}
K_2^{\mathrm{M}} (\C(X)) \xrightarrow{T} \displaystyle\bigoplus_{C\in X^{(1)}}\C(C)^\times \xrightarrow{\mathrm{div}}\displaystyle\bigoplus_{p\in X^{(2)}}\Z\cdot p 
\end{equation*}
Here $X^{(r)}$ denotes the set of closed subvarieties of $X$ of codimension $r$. 
The map $T$ denotes the tame symbol map from the Milnor $K_2$-group of the function field $\C(X)$.
\end{prop}
Therefore, a higher Chow cycle in $\mathrm{CH}^2(X,1)$ is represented by a formal sum
\begin{equation}\label{formalsum}
\sum_j (C_j, f_j)\in \displaystyle\bigoplus_{C\in X^{(1)}}\C(C)^\times
\end{equation}
where $C_j$ are prime divisors on $X$ and $f_j\in \C(C_j)^\times$ are non-zero rational functions 
on them such that $\sum_j {\mathrm{div}}_{C_j}(f_j) = 0$ as codimension 2
cycles on $X$.

A \textit{decomposable cycle} in $\mathrm{CH}^2(X,1)$ is an element of the image of the group homomorphism
\begin{equation}\label{intersectionproduct}
\mathrm{Pic}(X)\otimes_\Z \Gamma(X,\mathcal{O}_X^\times)=\CH^1(X)\otimes_\Z \CH^1(X,1) \longrightarrow \mathrm{CH}^2(X,1)
\end{equation}
induced by the intersection product. 
Let $C$ be a prime divisor on $X$ and $[C]\in \mathrm{Pic}(X)$ be the class corresponding to $C$. 
The image of $[C]\otimes \alpha$ under (\ref{intersectionproduct}) is represented by $(C,\alpha|_C)$ in the presentation (\ref{formalsum}). 
The cokernel of (\ref{intersectionproduct}) is denoted by $\mathrm{CH}^2(X,1)_{\mathrm{ind}}$. 
For $\xi \in \CH^2(X,1)$, $\xi_{\ind}$ denotes its image in $\mathrm{CH}^2(X,1)_{\mathrm{ind}}$. 
A cycle $\xi$ is called \textit{indecomposable} if $\xi_{\ind }\neq 0$.

\subsection{The regulator map}
By the canonical identification of $\mathrm{CH}^2(X,1)$ with the motivic cohomology $H^3_M(X,\Z(2))$, there exists the map
\begin{equation}\label{regulatormap}
\mathrm{reg}: \mathrm{CH}^2(X,1) \longrightarrow H^3_{\mathcal D}(X, \Z(2)) = \frac{H^2(X,\C)}{F^2H^2(X,\C)+H^2(X,\Z(2))}
\end{equation}
 called the \textit{regulator map}.
The target $H^3_{\mathcal D}(X, \Z(2))$ denotes the Deligne cohomology of $X$. 
This map can be regarded as the Abel-Jacobi map for $\CH^2(X,1)$. 
We recall an explicit formula for (\ref{regulatormap}) following \cite{Le} p.458--459. \par
Let $X$ be a $K3$ surface over $\C$. By the Poincar\'e duality, the Deligne cohomology of $X$ is isomorphic to the generalized complex torus
\begin{equation}\label{Delignefunctional}
H^3_\D(X, \Z(2)) \simeq \frac{(F^1H^2(X,\C))^\vee}{H_2(X,\Z)}
\end{equation}
where $(F^1H^2(X,\C))^\vee$ is the dual $\C$-vector space of $F^1H^2(X,\C)$ and we regard $H_2(X,\Z)$ as a subgroup of $(F^1H^2(X,\C))^\vee$ by the integration. 
Under the isomorphism (\ref{Delignefunctional}), the image of the cycle $\xi$ represented by a formal sum $\sum_j (C_j, f_j)$ under the regulator map is described as follows.

Let $D_j$ be the normalization of the closed curve $C_j$ on $X$. 
Let $\mu_j: D_j\rightarrow X$ denote the composition of $D_j\rightarrow C_j$ and $C_j\rightarrow X$. 
We will define a topological 1-chain $\gamma_j$ on $D_j$.
If $f_j$ is constant, we define $\gamma_j = 0$.
If $f_j$ is not constant, we regard $f_j$ as a finite morphism from $D_j$ to $\P^1$. 
Then we define $\gamma_j = f_j^{-1}([\infty, 0])$ where $[\infty, 0]$ is a path on $\P^1$ from $\infty$ to $0$ along the positive real axis. 
By the condition $\sum_j \div_{C_j}(f_j) = 0$, $\gamma = \sum_j (\mu_j)_*\gamma_j$ is a topological 1-cycle on $X$.
Since $H_1(X, \Z) = 0$, there exists a 2-chain $\Gamma$ on $X$ such that $\partial\Gamma = \gamma$.
In this paper, $\gamma$ and $\Gamma$ are called the \textit{1-cycle associated with $\xi$} and a \textit{2-chain associated with $\xi$}, respectively.
Then the image of $\xi$ under the regulator map is represented by the pairing
\begin{equation*}
\langle \mathrm{reg}(\xi), [\omega]\rangle = \displaystyle\int_\Gamma\omega  + \sum_j\dfrac{1}{2\pi\sqrt{-1}}\displaystyle\int_{D_j-\gamma_j}\log (f_j)\mu_j^*\omega\quad ([\omega]\in F^1H^2(X,\C)).
\end{equation*}
Here $\log(f_j)$ is the pull-back of the logarithmic function on $\P^1-[\infty,0]$ by $f_j$.

In this paper, we use the following variant of the regulator map.
\begin{defn}\label{transregdefn}
The \textit{transcendental regulator map} is the composite of the regulator map $(\ref{regulatormap})$ and the projection induced by $H^{2,0}(X) \hookrightarrow F^1H^2(X,\C)$. 
\begin{equation*}
\begin{tikzcd}
r: \CH^2(X,1) \arrow[r] &\dfrac{F^1H^2(X,\C)^\vee}{H_2(X,\Z)} \arrow[r,twoheadrightarrow] &\dfrac{H^{2,0}(X)^\vee}{H_2(X,\Z)}
\end{tikzcd}
\end{equation*}
We denote this map by $r$. 
\end{defn}

By taking the pairing with a non-zero holomorphic 2-form $\omega$ on $X$, 
we have an isomorphism $H^{2,0}(X)^\vee/H_2(X,\Z) \simeq \C/\Pc(\omega)$ where $\Pc(\omega)$ is the subgroup of $\C$ defined by 
\begin{equation*}
\Pc(\omega) = \left\{\displaystyle \int_{\Gamma}\omega  \in \C : \Gamma \text{ is a toplogical 2-cycles on }X.\right\}.
\end{equation*}
i.e. $\Pc(\omega)$ is the set of periods of $X$ with respect to $\omega$. 
By the above formula, the image of $\xi \in \CH^2(X,1)$ under the transcendental regulator map is 
\begin{equation}\label{transregval}
\langle r(\xi), [\omega]\rangle  \equiv \int_{\Gamma}\omega \mod \Pc(\omega).
\end{equation}
where $\Gamma$ is a 2-chain associated with $\xi$. 
If $\xi$ is decomposable, $r(\xi) = 0$. 
This implies the following.

\begin{prop}\label{transregprop}
If $r(\xi)\neq 0$, we have $\xi_{\ind}\neq 0$. In other words, the transcendental regulator map factors through $\CH^2(X,1)_{\ind}$.
\end{prop}

\subsection{A relative setting}
Since it is difficult to prove non-vanishingness of an element of $H^{2,0}(X)^\vee/H_2(X,\Z)$, we use its relative version.
Let $\pi:\X\rightarrow S$ be an algebraic family of $K3$ surfaces over a variety $S$. 
We define sheaves $\Pc$ and $\Qc$ of abelian groups on $S$ by 
\begin{equation*}
\begin{aligned}
\Pc &= \mathrm{Im}(R^2\pi_*\underline{\Z}_ {\X}\rightarrow \mathcal{H}om_{\O^{\an}_{S}}(\pi_*\Omega^2_{\X/S}, \O^\an_{S})) \\
\Qc &= \mathrm{Coker}(R^2\pi_*\underline{\Z}_ {\X}\rightarrow \mathcal{H}om_{\O^{\an}_{S}}(\pi_*\Omega^2_{\X/S}, \O^\an_{S})).
\end{aligned}
\end{equation*}
where $\underline{\Z}_{\X}$ is the constant sheaf on $\X$ and $\O^\an_{S}$ is the sheaf of holomorphic functions on $S$. 
Note that $\Pc$ is a local system on $S$.

By considering the pairing with a non-zero relative 2-form $\omega \in \Gamma(\X,\Omega^2_{\X/S})$,
we have an isomorphism $\mathcal{H}om_{\O^{\an}_{S}}(\pi_*\Omega^2_{\X/S}, \O^\an_{S})\simeq \O^\an_{S}$.
Hence $\Pc$ is isomorphic to the subsheaf $\Pc_\omega$ of $\O^\an_{S}$ generated by period functions with respect to $\omega$ and $\Qc$ is isomorphic to $\Qc_\omega = \O^\an_{S}/\Pc_\omega$.
For a local section $\varphi$ of $\O^\an_{S}$, its image in $\Qc_\omega$ is denoted by $[\varphi]$. 

For each $s\in S$, there exists the evaluation map 
\begin{equation*}
\ev_s: \Gamma(S,\Qc) \rightarrow H^{2,0}(\X_s)^\vee/H_2(\X_s,\Z).
\end{equation*}
Under the isomorphisms $\Qc\simeq \Qc_{\omega}$ and $H^{2,0}(\X_s)^\vee/H_2(\X_s,\Z)\simeq \C/\Pc(\omega_s)$ induced by $\omega$ and $\omega_s$, the evaluation map coincides with the map 
\begin{equation*}
\Qc_\omega = \O^\an_{S}/\Pc_\omega \ni [f] \longmapsto f(s)\mod \Pc(\omega_s)\quad \in \C/\Pc(\omega_s)
\end{equation*}
where $f(s)\in \C$ denotes the value of the holomorphic function $f$ at $s$.
The following elementary lemma is crucial for the result.
\begin{lem}\label{basiclem}
For a non-zero element $\nu \in \Gamma(S,\Qc)$, we have $\ev_s(\nu)\neq 0$ for very general $s\in S$.
\end{lem}
\begin{proof}
Since the question is local, we can shrink $S$ in the sense of the classical topology. 
By fixing a relative 2-form $\omega \in \Gamma(\X,\Omega_{\X/S}^2)$, we have the isomorphism $\Qc \simeq \O_S^\an/\Pc_\omega$. 
We may assume that there exist a holomorphic function $\varphi\in \O_S^{\an}(S)$ such that $\nu = [\varphi]$ and a free basis $f_1,f_2,\dots,f_r$ of $\Pc_\omega(S)$.

For each $\underline{c} = (c_i) \in \Z^r$, we define the holomorphic function $F_{\underline{c}}$ by 
\begin{equation*}
F_{\underline{c}} = \varphi - \sum_{i=1}^r c_if_i.
\end{equation*}
Since $\nu = [\varphi]$ is non-zero in $\Qc_\omega(S)$, $F_{\underline{c}}$ is a non-zero holomorphic function for each $\underline{c} \in \Z^r$.
Consider the countable family $\{F_{\underline{c}}\}_{\underline{c}\in \Z^r}$ of the holomorphic functions.
Then outside of the zeros of the functions in this family, $\varphi(s)\not\in \Pc(\omega_s) = \langle f_1(s),f_2(s),\dots, f_r(s) \rangle_\Z$.
Hence $\ev_s(\nu)\neq 0$.
\end{proof}

The following corollary is also used in technical propositions.

\begin{cor}\label{basiccor}
If local sections $\nu, \nu'$ of $\Qc$ on an open subset $U$ satisfy $\ev_s(\nu) = \ev_s(\nu')$ for any $s\in U$, we have $\nu=\nu'$. 
\end{cor}

Finally, we consider the regulator map in the relative setting.
Suppose that we have irreducible divisors $\Cc_j$ on $\X$ which are smooth over $S$ and non-zero rational functions $f_j$ on $\Cc_j$ whose zeros and poles are also smooth over $S$.
Assume that they satisfy the condition $\sum_j\div_{(\Cc_j)_s}((f_j)_s) = 0$ for each $s\in S$.
Then we have a family of higher Chow cycles $\xi = \{\xi_s\}_{s\in S}$ such that $\xi_s\in \CH^2(\X_s,1)$ is represented by the formal sum $\sum_{j}((\Cc_j)_s,(f_j)_s)$.
A family of higher Chow cycles constructed in this way is called an \textit{algebraic family of higher Chow cycles} in this paper.
\begin{rem}
In the above situation, let $\xi$ be the higher Chow cycle on the total space $\X$ defined by the formal sum $\sum_j(\Cc_j,f_j)$.
For each $s\in S$, $\xi_s \in \CH^2(\X_s,1)$ is the pull-back of $\xi$ by $\X_s\hookrightarrow \X$.
Hence we can regard an algebraic family of higher Chow cycle as a higher Chow cycle of the total space $\X$.
\end{rem}
If we shrink $S$ in the sense of the classical topology, there exists a $C^\infty$-family of topological 2-chains $\{\Gamma_s\}_{s\in S}$ such that $\Gamma_s$ is a 2-chain associated with $\xi_s$.
If we fix a relative 2-form $\omega$, the function 
\begin{equation*}
S\ni s \mapsto \int_{\Gamma_s}\omega_s \in \C
\end{equation*}
is holomorphic (cf. \cite{CL} Proposition 4.1).
Hence we can define the element $\nu_{\mathrm{tr}}(\xi) \in \Gamma(S,\Qc)$ such that 
\begin{equation*}
\ev_s(\nu_{\mathrm{tr}}(\xi)) = r(\xi_s)
\end{equation*}
for every $s\in S$.
This $\nu_{\mathrm{tr}}(\xi)$ can be regarded as a part of the normal function associated with $\xi$.
%We call it \textit{the transcendental part of the normal function associated with $\xi$}.

By combining Proposition \ref{transregprop} and Lemma \ref{basiclem}, we have the following.
\begin{prop}\label{mainproperty}
Let $\X\rightarrow S$ be an algebraic family of $K3$ surfaces and $\xi = \{\xi_s\}_{s\in S}$ be an algebraic family of higher Chow cycles.
Suppose $\nu_{\mathrm{tr}}(\xi)\neq 0$.
Then for very general $s\in S$, we have $(\xi_s)_{\ind} \neq 0$.
\end{prop}

 \section{The family of Kummer surfaces}
In this section, we define the family of Kummer surfaces. 
This is the famous family, and our main purpose is to fix the notation.
Let $S_0 = \P^1-\{0,1,\infty\}$ and $A_0= \C\left[c,\frac{1}{c(c-1)}\right]$ be the coordinate ring of $S_0$.
Let $\E\rightarrow S_0$ be the Legendre family of elliptic curves over $S_0$.
The fiber of $\E \rightarrow S_0$ over $c\in S_0$ is the elliptic curve 
\begin{equation*}
\E_c: y^2=x(1-x)(1-c x).
\end{equation*}
We have the natural $S_0$-morphism $\E \rightarrow \P^1\times S_0$ defined by $(x,y)\mapsto x$.
Let $\Sigma$ denote the set of 2-torsion sections of $\E\rightarrow S_0$.
The set $\Sigma$ consists of 4 sections corresponding to $x=0,1,1/c,\infty$.
We use the same symbol $\Sigma$ for its image under $\E\rightarrow \P^1\times S_0$.

Next, we define the family of Kummer surfaces and related families.
Let $T_0$ be the (Zariski) open set of $S_0\times S_0$ defined by
\begin{equation*}
T_0 = \left\{(a,b)\in S_0\times S_0: a\neq b, 1-b, \frac{1}{b}, \frac{1}{1-b},\frac{b-1}{b},\frac{b}{b-1}\right\}
\end{equation*}
and $B_0$ be the coordinate ring of $T_0$.
We will define the following families of varieties over $T_0$.
\begin{equation*}
\begin{tikzcd}
\widetilde{\A}_0 \arrow[r] \arrow[d,"\text{blowing-up}"{near start}, "\text{along }\Sigma^2"{near end}]  &[30pt] \widetilde{\X}_0 \arrow[r]\arrow[d,"\text{blowing-up}"{near start}, "\text{along }\Sigma^2"{near end}] &[30pt] \widetilde{\Y}_0 \arrow[d,"\text{blowing-up}"{near start}, "\text{along }\Sigma^2"{near end}]  \\
\A_0 \arrow[r,"2:1\text{ cover}"'] & \X_0 \arrow[r,"2:1\text{ cover}"'] & \Y_0
\end{tikzcd}
\end{equation*}
Let $\Y_0=\P^1\times \P^1\times T_0$.
We use $x$ (resp. $y$) for coordinates of the first (resp. second) $\P^1$ of $\Y_0$.
We define the family of abelian surfaces $\A_0\rightarrow T_0$ by the restriction of $\E \times \E \rightarrow S_0\times S_0$ to $T_0$.
The fiber of $\A_0\rightarrow T_0$ over $(a,b)\in T_0$ corresponds to the product of elliptic curves $\E_a\times \E_b$.
We have the natural $T_0$-morphism $\A_0\rightarrow \Y_0$ induced by the direct product of $S_0$-morphism $\E \rightarrow \P^1\times S_0$.
The map $\A_0\rightarrow \Y_0$ is a 4:1 cover.
Let $\iota: \A_0\rightarrow \A_0$ be the map of taking inverses and we define $\X_0$ as the quotient of $\A_0$ by $\iota$.
Then the morphism $\A_0\rightarrow \Y_0$ factors $\X_0$.

We define $\Sigma^2$ as the set of 2-torsion sections of $\A_0 \rightarrow T_0$.
The set $\Sigma^2$ consists of 16 sections corresponding to $\Sigma\times \Sigma$. 
We use the same symbol $\Sigma^2$ for its images in $\X_0$ and $\Y_0$.
Using the coordinates $x$ and $y$, the set $\Sigma^2$ can be described as 16 points
\begin{equation*}
(x,y) \in \{0,1,1/a,\infty\}\times \{0,1,1/b,\infty\}.
\end{equation*}

Let $\widetilde{\A}_0, \widetilde{\X}_0$ and $\widetilde{\Y}_0$ be the blowing-ups of $\A_0$, $\X_0$ and $\Y_0$ along $\Sigma^2$.
The universality of the blowing-up induces the maps $\widetilde{\A}_0\rightarrow \widetilde{\X}_0\rightarrow \widetilde{\Y}_0$.
Since $\widetilde{\X}_0\rightarrow \X_0$ is the minimal resolution of singularities, $\widetilde{\X}_0\rightarrow T_0$ is the family of Kummer surfaces.
At $t=(a,b)\in T_0$, the fiber $\widetilde{\X}_t$ is the Kummer surface $\Km(\E_a\times\E_b)$ associated with $\E_a\times \E_b$.
For $\sigma\in \Sigma^2$, $Q_{\sigma}$ denotes the exceptional divisor over $\sigma\in \Sigma^2$.
The configuration of $Q_\sigma$ on $\widetilde{\X}_0$ is described in Figure \ref{Qsigmapicture}. 
\begin{figure}[h]
\centering
\includegraphics[width=9cm]{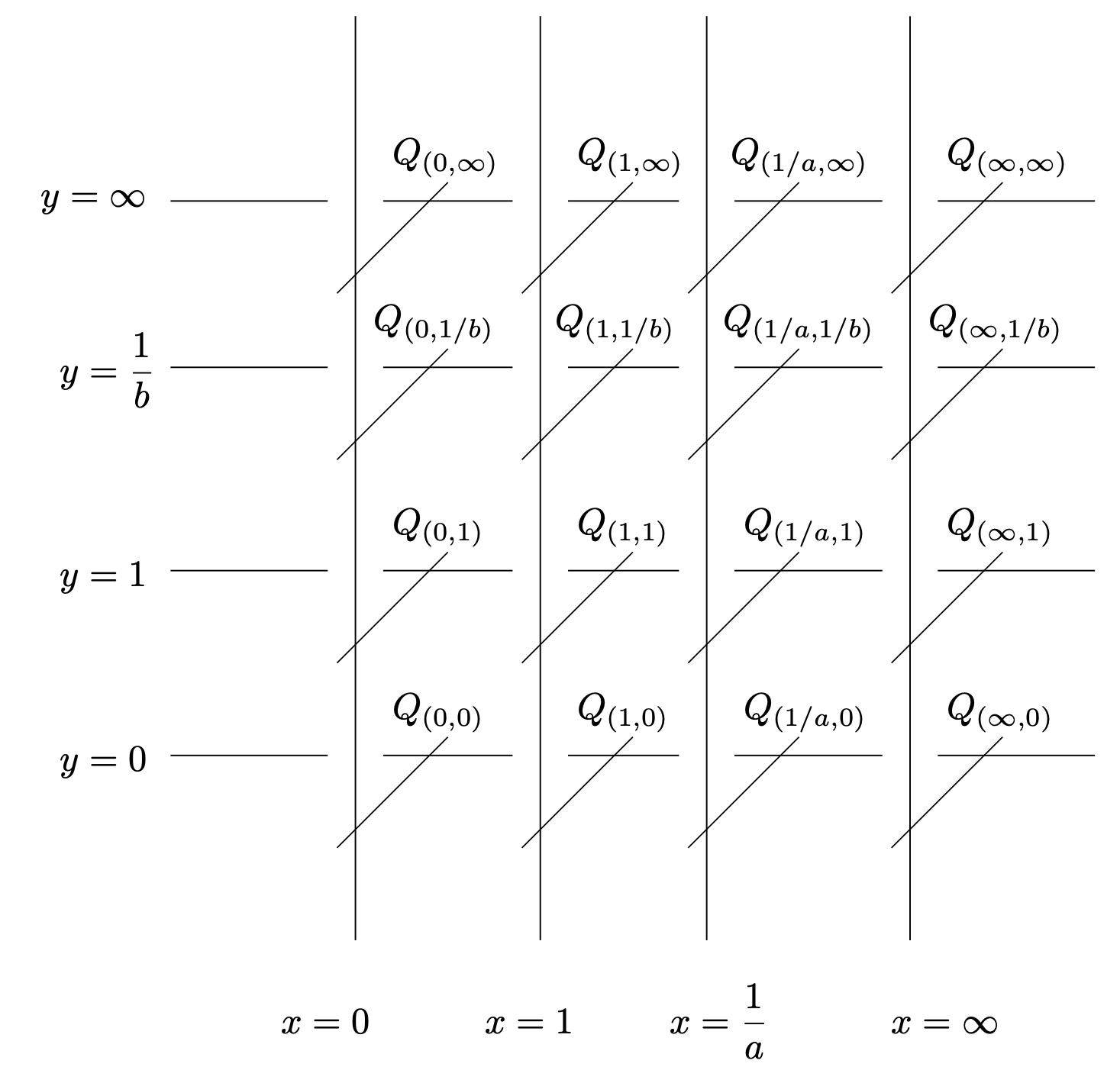}
\caption{The exceptional divisors $Q_\sigma$ on $\widetilde{\X}_0$}\label{Qsigmapicture}
\end{figure} 

Finally, we define local charts on $\Y_0$ and $\widetilde{\X}_0$.
Let $U=\Spec B_0[x,y]\subset \Y_0$ be the affine open subset which is the complement of the divisors $x=\infty$ and $y=\infty$.
The inverse image of $U$ by $\widetilde{\X}_0\rightarrow \Y_0$ is covered by two open affine subschemes $V=\Spec B_0[x,y,v]/(v^2f(x)-g(y))$ and $W=\Spec B_0[x,y,w]/(w^2g(y)-f(x))$ where $f(x)=x(1-x)(1-ax)$ and $g(y)=y(1-y)(1-by)$.
These two subsets are glued by the relation $v = 1/w$.

\section{Construction of initial higher Chow cycles}
In this section, we construct families $\xi_0, \xi_1$ and $\xi_{\infty}$ of higher Chow cycles on a base change $\widetilde{\X}\rightarrow T$ of $\widetilde{\X}_0\rightarrow T_0$.
\subsection{Construction of higher Chow cycles at fibers}
First, we explain the construction at each fiber over $t=(a,b)\in T_0$.
Let $\D \subset \P^1\times \P^1=(\Y_0)_t$ be the diagonal curve. 
Since $a\neq b$, $\D $ intersects with the branching locus of $(\X_0)_t\rightarrow (\Y_0)_t$ as Figure \ref{Dpicture}.

 \begin{figure}[h]
\centering
\includegraphics[width=8cm,pagebox=cropbox,clip]{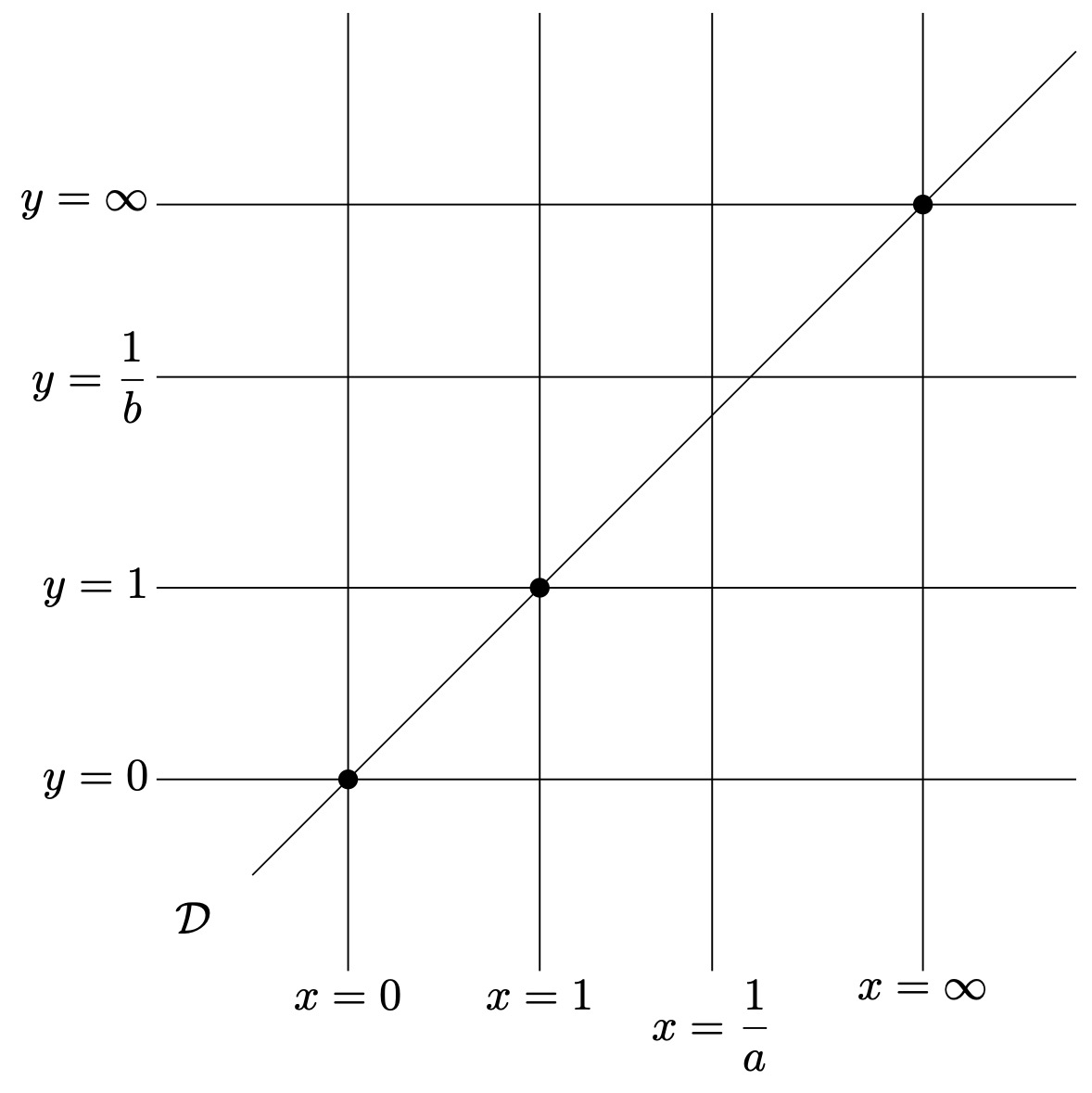}
\caption{The diagonal curve $\D$ and the branching locus}
\label{Dpicture}
\end{figure}

Let $\Cc \subset (\X_0)_t$ be the pull-back of $\D$ by $(\X_0)_t\rightarrow (\Y_0)_t$ and $\widetilde{\Cc}$ be the strict transform of $\Cc$ by the blowing-up $(\widetilde{\X}_0)_t\rightarrow (\X_0)_t$.
Then we see that $\widetilde{\Cc}$ is smooth and $\widetilde{\Cc} \rightarrow \D $ is a double covering ramified at 2 points.
Hence $\widetilde{\Cc}$ is isomorphic to $\P^1$.
Furthermore, for each $\bullet \in \{0,1,\infty\}$, $\widetilde{\Cc} $ intersects with $Q_{(\bullet,\bullet)}$ at 2 points $p_{\bullet}^+$ and $p_{\bullet}^-$ (cf. Figure \ref{curvefigure}).

\begin{figure}[h]
\centering
\includegraphics[width=9cm,pagebox=cropbox,clip]{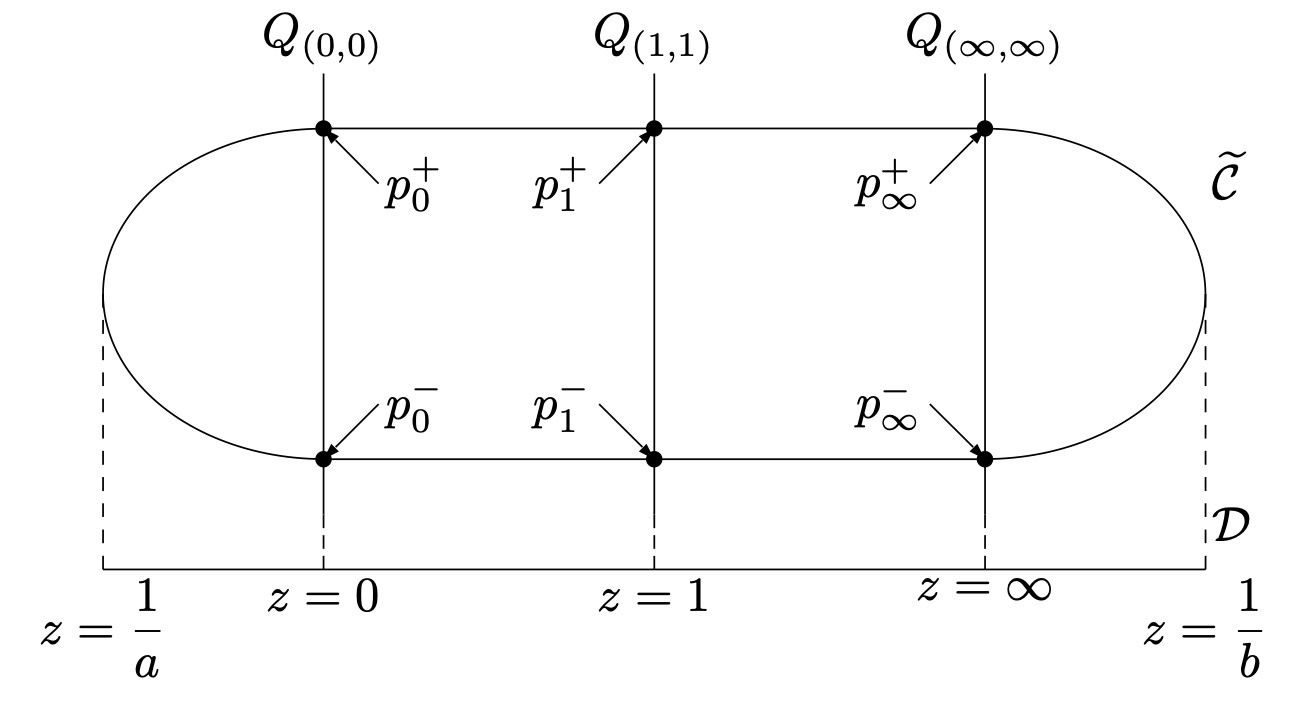}
\caption{The intersections of $\widetilde{\Cc}$, $Q_{(0,0)}$, $Q_{(1,1)}$ and $Q_{(\infty,\infty)}$}
\label{curvefigure}
\end{figure}

Hence we can find rational functions $\psi_0, \psi_1,\psi_\infty \in \C(\widetilde{\Cc})$ and $\varphi_{\bullet} \in \C(Q_{(\bullet, \bullet)})\:\: (\bullet = 0,1,\infty\})$ which satisfy the following relations.
\begin{equation}\label{phipsi}
\begin{aligned}
&\div_{\widetilde{\Cc}}(\psi_0) = p_{0}^- - p_{0}^+ = - \div_{Q_{(0,0)}}(\varphi_0) \\
&\div_{\widetilde{\Cc}}(\psi_1) = p_{1}^- - p_{1}^+ = - \div_{Q_{(1,1)}}(\varphi_1) \\
&\div_{\widetilde{\Cc}}(\psi_\infty) = p_\infty^- - p_\infty^+ = - \div_{Q_{(\infty,\infty)}}(\varphi_\infty) 
\end{aligned}
\end{equation}
We define $(\xi_0)_t,(\xi_1)_t,(\xi_\infty)_t\in \CH^2((\widetilde{\X}_0)_t,1)$ by the formal sums
\begin{equation}\label{familyXican}
\begin{aligned}
&(\xi_0)_t =  (\widetilde{\Cc}, \psi_0) + (Q_{(0,0)},\varphi_0), \\
&(\xi_1)_t=  (\widetilde{\Cc}, \psi_1) + (Q_{(1,1)},\varphi_1), \\
&(\xi_\infty)_t = (\widetilde{\Cc}, \psi_\infty) + (Q_{(\infty,\infty)},\varphi_\infty).
\end{aligned}
\end{equation}
\subsection{Construction of families of higher Chow cycles}
To get families of higher Chow cycles, it is enough to construct rational functions $\varphi_\bullet$ and $\psi_\bullet$ on the family.
However, the intersection points $p_{1}^+$ and $ p_{1}^-$ (resp. $p_{\infty}^+$ and $ p_{\infty}^-$) interchange by the monodromy of $T_0$.
Hence it is impossible to construct such rational functions for $\bullet = 1,\infty$.
Thus it is necessary to take a finite \'etale base change of $T_0$ to define the families of cycles.

Let $B = B_0[\sqrt{a},\sqrt{b},\sqrt{1-a},\sqrt{1-b}]$ and $T\rightarrow T_0$ be the finite \'etale cover corresponding to $B_0\rightarrow B$. The base changes of $\Y_0,\X_0,\A_0, \widetilde{\Y}_0, \widetilde{\X}_0$ and $\widetilde{\A}_0$ by $T\rightarrow T_0$ are denoted by $\Y,\X,\A, \widetilde{\Y}, \widetilde{\X}$ and $\widetilde{\A}$, respectively.

Let $\D\subset \Y$ be the closed subscheme defined by the local equation $x=y$, $\Cc$ be its pull-back by $\X\rightarrow \Y$ and $\widetilde{\Cc}$ be its strict transform by the blowing up $\widetilde{\X}\rightarrow \X$.

On the local chart $V$, $\widetilde{\Cc}\hookrightarrow \widetilde{\X}$ is described by the following ring homomorphism.
\begin{equation*}
B[x,y,v]/(v^2f(x)-g(y))  \lra B[z,v]/(v^2(1-az)-(1-bz));  x,y,v\mapsto z,z,v
\end{equation*}
By this description, we see that $Q_{(0,0)}$ and $\widetilde{\Cc}$ intersect at $(x,y,v) = (0,0,\pm 1)$, and $Q_{(1,1)}$ and $\widetilde{\Cc}$ intersect at
$(x,y,v) = \left(1,1,\pm \sqrt{1-b}/\sqrt{1-a}\right)$.
By the local computation on another local chart containing $(\infty,\infty)\in \Sigma^2$, $Q_{(\infty,\infty)}$ and $\widetilde{\Cc}$ intersect at $(\xi,\eta,v')=\left(0,0,\pm \sqrt{b}/\sqrt{a}\right)$ where local coordinates $\xi,\eta$ and $v'$ are defined by $\xi=1/x$, $\eta =1/y$ and $v' = x^2v/y^2$.
Hence we can define rational functions $\psi_\bullet \in \C(\widetilde{\Cc})$ and $\varphi_{\bullet}\in \C(Q_{(\bullet,\bullet)})\: (\bullet \in \{0,1,\infty\})$ by the following equations.
\begin{equation*}
\begin{aligned}
& \psi_0 = (v+1)\cdot(v-1)^{-1} && \varphi_0 = (v-1)\cdot(v + 1)^{-1} \\
& \psi_1 = \left(v + \frac{\sqrt{1-b}}{\sqrt{1-a}}\right)\cdot\left(v - \frac{\sqrt{1-b}}{\sqrt{1-a}}\right)^{-1} && \varphi_1 = \left(v-\frac{\sqrt{1-b}}{\sqrt{1-a}}\right)\cdot\left(v+\frac{\sqrt{1-b}}{\sqrt{1-a}}\right)^{-1} \\
& \psi_\infty = \left(v' + \frac{\sqrt{b}}{\sqrt{a}}\right)\cdot\left(v' - \frac{\sqrt{b}}{\sqrt{a}}\right)^{-1} && \varphi_\infty =\left(v'- \frac{\sqrt{b}}{\sqrt{a}}\right)\cdot\left(v'+ \frac{\sqrt{b}}{\sqrt{a}}\right)^{-1}
\end{aligned}
\end{equation*}
They satisfy the relations in (\ref{phipsi}) and we define algebraic families of higher Chow cycles $\xi_0 = \left\{(\xi_0)_t\right\}_{t\in T}, \xi_1 = \left\{(\xi_1)_t\right\}_{t\in T}$ and $\xi_{\infty} =\left\{(\xi_\infty)_t\right\}_{t\in T}$ by the equations (\ref{familyXican}).

\section{Computation of the regulator}
In this section, we compute the image of $(\xi_1)_t-(\xi_0)_t$ under the regulator map and prove its indecomposability.
Our main result is the following.
\begin{thm}\label{mainthm1}
For very general $t\in T$, $(\xi_1)_{t}-(\xi_0)_{t}$ is an indecomposable cycle.
\end{thm}
The main ingredients of the proof is the following two propositions.

\begin{prop}\label{preprop1}
Let $\omega$ be the relative 2-form $\frac{dx\wedge dy}{vf(x)}$ on the family $\widetilde{\X}\rightarrow T$ and $\Dc_1,\Dc_2: \O_{T}^\an \rightarrow \O_{T}^\an$ be the differential operators defined by 
\begin{equation*}
\begin{aligned}
&\Dc_1 = a(1-a) \frac{\del^2}{\del a^2} +(1-2a)\frac{\del }{\del a} -\frac{1}{4} \\
&\Dc_2 = b(1-b) \frac{\del^2}{\del b^2} +(1-2b)\frac{\del }{\del b} -\frac{1}{4}.
\end{aligned}
\end{equation*}
Let $\Dc = \begin{pmatrix} \Dc_1 \\ \Dc_2\end{pmatrix}:\O_{T}^\an \rightarrow \left(\O_{T}^\an\right)^{\oplus 2}$.
Then for any local section $f$ of $\Pc_{\omega}\subset \O_{T}^\an$, we have $\Dc(f)=0$.
In particular, $\Dc$ factors the sheaf $\Qc_{\omega}$. 
We use the same symbol $\Dc$ for the induced morphisms $\Qc_\omega \rightarrow \left(\O_{T}^\an\right)^{\oplus 2}$ and $\Qc\simeq \Qc_\omega \rightarrow \left(\O_{T}^\an\right)^{\oplus 2}$.
\end{prop}
Recall that $\Pc_\omega$ is the local system consisting of period functions with respect to $\omega$ and $\Qc_\omega$ is the quotient of $\O_{T}^\an$ by $\Pc_\omega$. 
The differential operator $\Dc$ is called a \textit{Picard-Fuchs differential operator} because it annihilates all period functions with respect to $\omega$.

By Proposition \ref{mainproperty}, to prove Theorem \ref{mainthm1}, it is enough to show $\nu_{\tr}(\xi_1-\xi_0)$ is non-zero.
Then by the Proposition \ref{preprop1}, it is enough to show $\Dc(\nu_{\tr}(\xi_1-\xi_0))$ is non-zero.
To prove this, we will find an explicit multivalued function which represents $\nu_{\tr}(\xi_1-\xi_0)$.

For $(a,b)\in T_0$ such that $a,b\in \R_{<0}$, let $\Lc(a,b)$ be the improper integral
\begin{equation}\label{functionL}
\Lc(a,b)=2\int_{\triangle}\frac{dxdy}{\sqrt{x(1-x)(1-ax)}\sqrt{y(1-y)(1-by)}}
\end{equation}
where $\triangle =\{(x,y)\in \R^2\::\: 0<y<x<1\}$.
This integral converges and defines a local holomorphic function around $(a,b)$. 
There exists a lift $\Lc$ of $\Lc(a,b)$ by $T\rightarrow T_0$ which satisfies the the following properties.
\begin{prop}\label{preprop5}
\begin{enumerate}
\item For any $t\in T$, $\Lc$ can be analytically continued to an open neighborhood of $t$.
\item The multivalued holomorphic function $\Lc$ represents $\nu_{\tr}(\xi_1-\xi_0)$ under the isomorphism $\Qc_\omega\simeq \Qc$ induced by the relative 2-form $\omega$ in Proposition \ref{preprop1}.
In particular, for any $t\in T$, we have 
\begin{equation}\label{transregK}
\langle r((\xi_1)_t-(\xi_0)_t), [\omega_t]\rangle \equiv \Lc(t) \mod \Pc(\omega_t)
\end{equation}
where $\omega_t$ is the pull-back of $\omega$ at $\widetilde{\X}_t$.
\item The multivalued holomorphic function $\Lc$ satisfies the following system of the differential equations.
\begin{equation}\label{diffL}
\Dc(\Lc)=
\frac{2}{a-b}
 \begin{pmatrix}
 \frac{\sqrt{1-b}}{\sqrt{1-a}}-1\\[1.5ex]
1-\frac{\sqrt{1-a}}{\sqrt{1-b}}
\end{pmatrix}.
\end{equation}
In particular, $\Dc(\nu_{\tr}(\xi_1-\xi_0))$ coincides with the right-hand side of (\ref{diffL}).
\end{enumerate}
\end{prop}

Thus we have $\nu_{\tr}(\xi_1-\xi_0)\neq 0$ and this implies Theorem \ref{mainthm1} by Proposition  \ref{mainproperty}.
We will prove Proposition \ref{preprop1} and Proposition \ref{preprop5} in this section.

\subsection{The Picard-Fuchs differential operator}
For $c\in \C-\R_{\ge 0}$, we consider the following improper integrals.
\begin{equation*}
P_1(c)=\int_0^1\frac{dx}{\sqrt{x(1-x)(1-c x)}},\quad P_2(c)=\int_1^\infty\frac{dx}{\sqrt{x(1-x)(1-c x)}}
\end{equation*}
As is well-known (see, e.g., \cite{WW}, p.~253), these integrals converge and give linearly independent solutions of the hypergeometric differential equation
\begin{equation}\label{diffopL}
c(1-c)\frac{d^2P}{dc^2} +(1-2c)\frac{dP}{dc} -\frac{1}{4}P=0.
\end{equation}
For any $c\in S_0$, they can be analytically continued to a neighborhood of $c$ and we use the samge notation for the resulting multivalued functions.

Let $\theta$ be the relative 1-form $\frac{dx}{y}$ on the Legendre family of elliptic curves $\E\rightarrow S_0$. 
For $c\in \C-\R_{\ge 0}$, let $\gamma_+,\gamma_-$ (resp. $\delta_+,\delta_-$) be lifts of paths $[0,1]$ and $[1,\infty]$ on $\P^1$ by $\E_c  \rightarrow  \P^1$.
Let $\theta_c$ be the pull-back of $\theta$ at $\E_c$. Then we have 
\begin{equation*}
\int_{\gamma_+}\theta_c = -\int_{\gamma_-}\theta_c,\quad \int_{\delta_+}\theta_c = -\int_{\delta_-}\theta_c
\end{equation*}
and they coincide with $\pm P_1(c)$ and $\pm P_2(c)$, respectively.
Since $[\gamma_+]-[\gamma_-]$ and $[\delta_+]-[\delta_-]$ are generators of $H_1(\E_c,\Z)$ (see, e.g., \cite{CMP}, p.~10), $2P_1$ and $2P_2$ are local basis of the period functions of $\E\rightarrow S_0$ with respect to $\theta$.
In particular, hypergeometric differential equation (\ref{diffopL}) is a Picard-Fuchs differential equation of $\E\rightarrow S_0$.

Next, we will find a Picard-Fuchs differential operator of $\widetilde{\X}\rightarrow T$. Recall that we have $T$-morphisms
\begin{equation*}
\begin{tikzcd}
\A  &[30pt] \widetilde{\A} \arrow[l,"p"',"\text{blowing-up}"] \arrow[r,"\pi","2:1\text{ quotient}"'] &[30pt] \widetilde{\X}.
\end{tikzcd}
\end{equation*}
We name the morphisms $p$ and $\pi$ as above. 
We have the relative 2-form $pr_1^*(\theta)\wedge pr_2^*(\theta)$ on $\A\rightarrow T$ where $pr_i$ is the morphism $\A\rightarrow \A_0 \hookrightarrow \E\times \E \xrightarrow{pr_i} \E$.
Then its pull-back by $p$ is stable under the covering transformation of $\pi$, so it descends to $\widetilde{\X}$. We denote the resulting 2-form on $\widetilde{\X}$ by $\omega$.
The relative 2-form $\omega$ is described as $\frac{dx\wedge dy}{vf(x)}$ and $\frac{dx\wedge dy}{wg(y)}$ on the local charts $V$ and $W$.
Then we have the following.
\begin{prop}\label{perprop}
The four local holomorphic functions $2P_i(a)P_j(b) \:\: (i,j\in \{1,2\})$ are a local basis for the local system $\Pc_{\omega}\subset \O_{T}^\an$ generated by period functions of $\widetilde{\X}\rightarrow T$ with respect to $\omega$.
\end{prop}
\begin{proof}
By the K\"unneth formula, we see that $4P_i(a)P_j(b) \:\: (i,j\in \{1,2\})$ is a local basis for the local system generated by period functions of $\A\rightarrow T$ with respect to $pr_1^*(\theta)\wedge pr_2^*(\theta)$.
For $t\in T$, let $\phi$ be the morphism of Hodge structures defined by 
\begin{equation*}
\begin{tikzcd}
\phi:&[-28pt] H^2(\A_t) \arrow[r,"p^*"] & H^2(\widetilde{\A}_t) \arrow[r,"\pi_!"] & H^2(\widetilde{\X}_t)
\end{tikzcd}
\end{equation*}
where $\pi_!$ is the Gysin morphism induced by $\pi$.
Since the mapping degree of $\pi$ is 2, $\pi_!\circ \pi^*: H^2(\widetilde{\X}_t)\rightarrow H^2(\widetilde{\X}_t)$ equals multiplication by 2 (cf. \cite{Voi02}, Remark 7.29).
Since $\pi^*(\omega) = p^*\left(pr_1^*(\theta)\wedge pr_2^*(\theta)\right)$, we have the relation 
\begin{equation}\label{2mulomega}
\phi([\left(pr_1^*(\theta)\wedge pr_2^*(\theta)\right)_t] = 2[\omega_t].
\end{equation}
Let $\phi^\vee :  H_2(\widetilde{\X}_t) \rightarrow H_2(\A_t)$ be the dual of $\phi$. For any $[\Gamma]\in  H_2(\widetilde{\X}_t,\Z)$ and $[\Gamma']\in H_2(\A_t,\Z)$ such that $\phi^\vee([\Gamma]) = [\Gamma']$, we have 
\begin{equation*}
\int_{\Gamma}\omega_t = \frac{1}{2}\int_{\Gamma'}\left(pr_1^*(\theta)\wedge pr_2^*(\theta)\right)_t
\end{equation*}
by (\ref{2mulomega}).
Since the right-hand side is a linear combination of $2P_i(a)P_j(b)$, we see that any period function of $\widetilde{\X}\rightarrow T$ with respect to $\omega$ is a linear combination of $2P_i(a)P_j(b)$.
Furthermore, since $\phi$ is injective and its cokernel has no torsion (cf. \cite{BPHV}, {\rm Chapter VIII, Proposition 5.1 and Corollary 5.6}), $\phi^\vee$ is surjective.
Thus $2P_i(a)P_j(b)$ itself is a period function for $i,j\in \{1,2\}$ and we have the result.
\end{proof}

Now we can prove Proposition \ref{preprop1}.

\begin{proof}[(Proof of Proposition \ref{preprop1})]
Let $f\in \Pc_\omega\subset \O_T^\an$ be any local section.
Then $f$ is a linear combination of $2P_i(a)P_j(b)$ by Proposition \ref{perprop}.
Since $P_i(c)$ are solutions of (\ref{diffopL}), both of the differential operators $\Dc_1$ and $\Dc_2$ annihilate $2P_i(a)P_j(b)$.
Hence we have the result.
\end{proof}

\subsection{Calculation of the regulator}
For a while, we fix $t_0\in T$ such that\footnote{Note that choosing a point on $T$ is equivalent to choosing a point $(a,b)\in T_0$ and branches of $\sqrt{a},\sqrt{b},\sqrt{1-a},\sqrt{1-b}\in \C$.} $\sqrt{1-a}$, $\sqrt{1-b}\in \R_{>0}$.
Our first goal is to compute the image of $(\xi_1)_t-(\xi_0)_t$ under the transcendental regulator map locally around $t_0$.
For the computation, we construct a 2-chain and use (\ref{transregval}).
In the calculation, we replace a 2-chain associated with $(\xi_1)_t-(\xi_0)_t$ with another 2-chain to make the computation easier.

Let $\xi$ be a higher Chow cycle on a $K3$ surface $X$ and $\gamma$ be the topological 1-cycle associated with $\xi$.
A topological 2-chain $\Gamma'$ is called a \textit{2-chain associated with $\xi$ in a weak sense} if there exists a finite family of curves $\{E_k\}_{k}$ on $X$ and a topological 2-chain $\Gamma_k$ on each $E_k$ such that 
\begin{equation}\label{weaksense}
\del \Gamma' = \gamma + \sum_{k}\del\Gamma_k.
\end{equation}
Let $\Gamma$ be a 2-chain associated with $\xi$. If $\Gamma'$ is a 2-chain associated with $\xi$ in a weak sense, $\Gamma$ and $\Gamma'+\sum_{k}\Gamma_k$ coincide up to topological 2-cycles. 
Since the pull-back of a holomorphic 2-form $\omega$ on $E_k$ vanishes, we have 
\begin{equation*}
\int_{\Gamma'}\omega \equiv \int_{\Gamma}\omega \underset{(\ref{transregval})}{\equiv} \langle r(\xi), [\omega]\rangle  \mod \Pc(\omega).
\end{equation*}
Hence we can also use $\Gamma'$ for the computation of the transcendental regulator map.

We will construct a desired 2-chain.
For $t$ which is sufficiently close to $t_0$, let $\triangle_+$ and $\triangle_-$ be the images of the following maps.
\begin{equation}\label{Deltamap}
\begin{tikzcd}[row sep = tiny]
\triangle = \{(x,y)\in \R^2: 0<y<x<1\} \arrow[r] &[-50pt]V &[-100pt] (\subset \widetilde{\X}_t) \\
(x,y) \arrow[r,mapsto] \aru &(x,y,v) = \left(x,y,\pm \frac{\sqrt{y(1-y)(1-by)}}{\sqrt{x(1-x)(1-ax)}}\right) \aru
\end{tikzcd}
\end{equation}
Note that since $t$ is sufficiently close to $t_0$, we may assume the function $\sqrt{1-ax}$ (resp. $\sqrt{1-by}$) do not ramify on $\triangle$ and we can fix the branch of it so that it takes a value $1$ and $\sqrt{1-a}$ at $x=0,1$ (resp. $1$ and $\sqrt{1-b}$ at $y=0,1$). 

We define $K_+$ and $K_-$ as the closures of $\triangle_+$ and $\triangle_-$, respectively. 
Let $\gamma$ be a path $[0,1]$ on $\P^1$ reparameterized so that $\gamma(s) = s^2$ (resp. $1-\gamma(s) = (1-s)^2$) on a neighborhood of $0$ (resp. $1$).
Then if we replace $x$ and $y$ in the target in (\ref{Deltamap}) by $\gamma(x)$ and $\gamma(y)$, respectively, (\ref{Deltamap}) can be extended to a map from a compact oriented manifold with corners, so $K_+$ and $K_-$ are $C^\infty$-chains. 

\begin{prop}\label{2chainprop}
For any $t$ which is sufficiently close to $t_0$, we have 
\begin{equation*}
\langle r((\xi_1)_t-(\xi_0)_t), [\omega_t]\rangle \equiv \int_{K_+}\omega_t -\int_{K_-}\omega_t \mod \Pc(\omega_t).
\end{equation*}
\end{prop}

To prove this, we should examine the boundaries of $K_+$ and $K_-$.
Recall that we have the following morphisms.
\begin{equation*}
\begin{tikzcd}
\widetilde{\X}_t \arrow[r,"2:1\text{ quotient}"] &[30pt] \widetilde{\Y}_t \arrow[r,"\text{blowing-up}"] &[30pt] \Y_t=\P^1\times \P^1
\end{tikzcd}
\end{equation*}
We regard $\triangle=\{(x,y)\in \R^2:0<y<x<1\}$ as a subset of $\Y_t=\P^1\times \P^1$.
Let $K$ be the closure of the inverse image of $\triangle$ by $\widetilde{\Y}_t\rightarrow \Y_t$ (see Figure \ref{Kfigure1}).
We define paths $\gamma_c,\gamma_{11},\gamma_{y},\gamma_{10}, \gamma_x$ and $\gamma_{00}$ on the boundary $\del K$ as in Figure \ref{Kfigure1}.
\begin{figure}[h]
\includegraphics[width=10cm,pagebox=cropbox,clip]{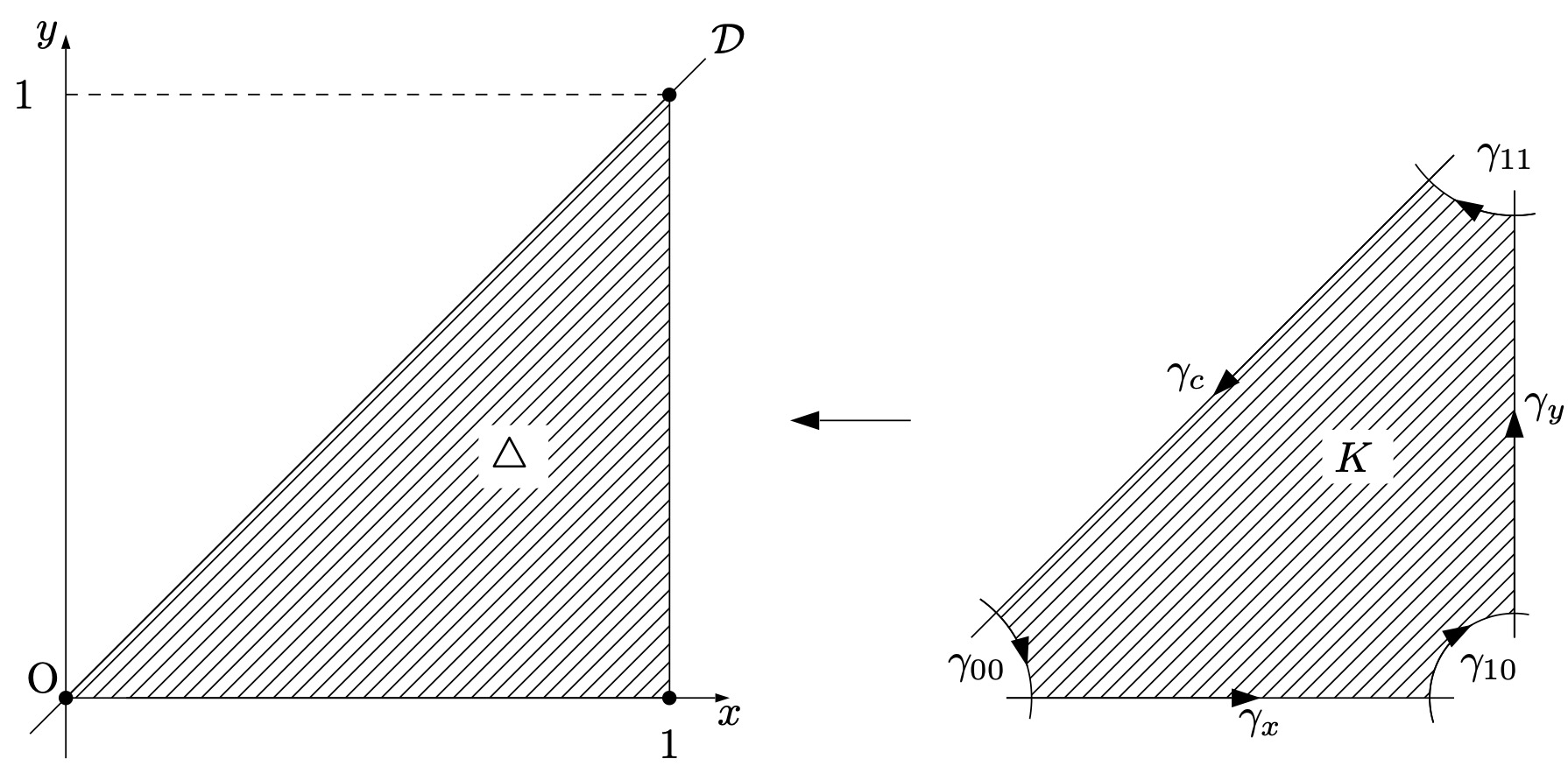}
\caption{The 2-chain $K$ and its boundary}
\label{Kfigure1}
\end{figure}

Then $\widetilde{\X}_t\rightarrow \widetilde{\Y}_t$ induces the homeomorphism $K_+\xrightarrow{\sim} K$ and $K_-\xrightarrow{\sim} K$ (see Figure \ref{Kfigure2}).
\begin{figure}[h]
\includegraphics[width=13cm,pagebox=cropbox,clip]{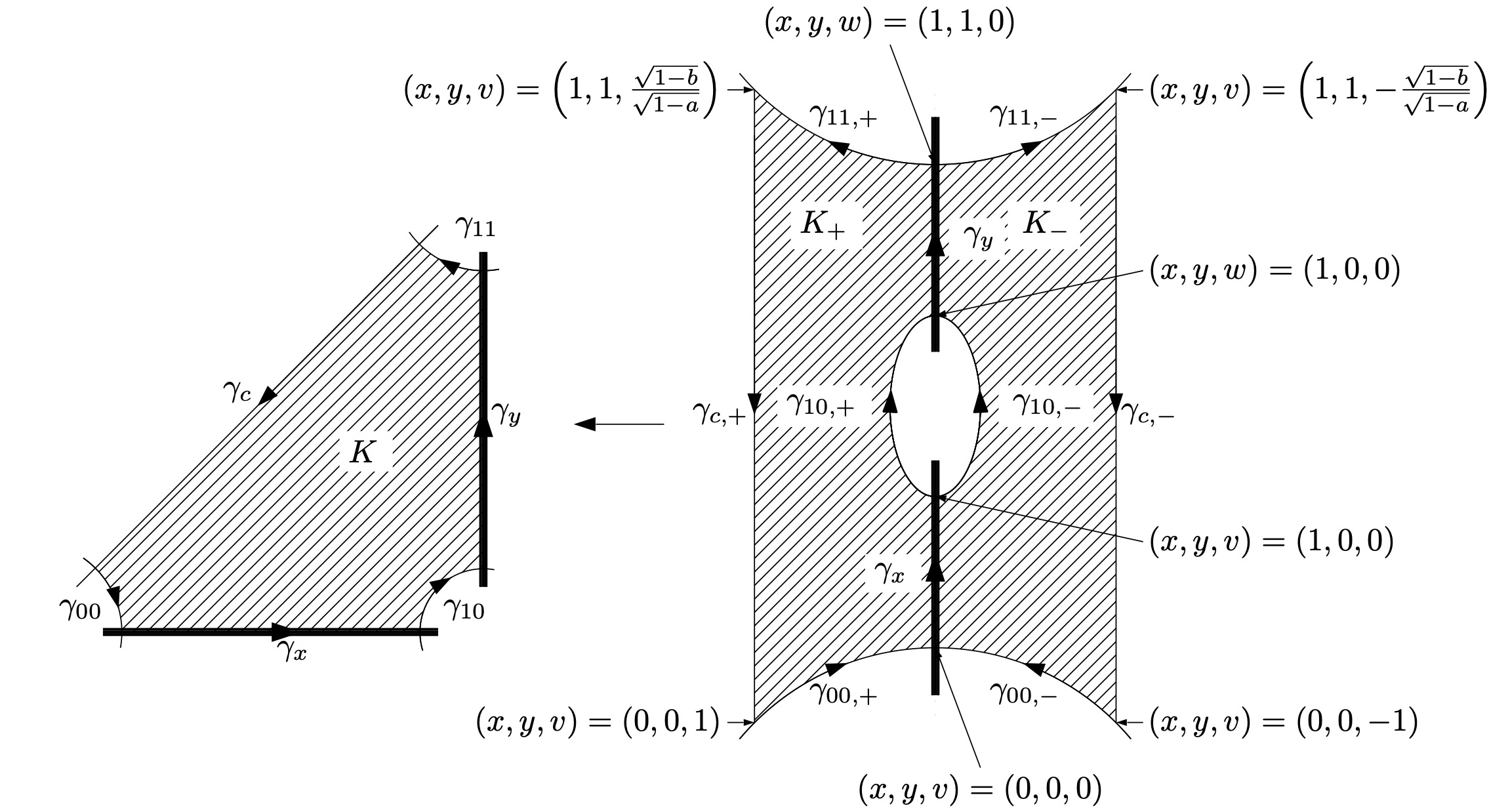}
\caption{The 2-chains $K_+$ and $K_-$ and their boundaries}
\label{Kfigure2}
\end{figure}
Here the bold line in the figure denotes the branching locus of $\widetilde{\X}_t\rightarrow \widetilde{\Y}_t$.
We define paths $\gamma_{c,\pm}, \gamma_{11,\pm}, \gamma_y,\gamma_{10,\pm},\gamma_x$ and $\gamma_{00,\pm}$ on $\del K_+$ and $\del K_-$ as in Figure \ref{Kfigure2}.
They satisfy the following properties.
\begin{enumerate}
\item The paths $\gamma_{c,+}$ and $\gamma_{c,-}$ are the lifts of $\gamma_c$.
Since $\gamma_c$ is on the strict transformation of $\D$, they are on the curve $\widetilde{\Cc}$.
\item Furthermore, since $t$ is close to $t_0$, $\gamma_{c,+}$ (resp. $\gamma_{c,-}$) is a path from \\
$(1,1,\sqrt{1-b}/\sqrt{1-a})$ to $(0,0,1)$ (resp. $(1,1,-\sqrt{1-b}/\sqrt{1-a})$ to $(0,0,-1)$).
\item The paths $\gamma_{00,+},\gamma_{10,+}$ and $\gamma_{11,+}$ (resp. $\gamma_{00,-},\gamma_{10,-}$ and $\gamma_{11,-}$) are the lifts of $\gamma_{00},\gamma_{10}$ and $\gamma_{11}$ and they are on the exceptional curves $Q_{(0,0)}, Q_{(1,0)}$ and $Q_{(1,1)}$, respectively.
\item Since $\gamma_x$ and $\gamma_y$ on $\del K$ is contained in the branching locus of $\widetilde{\X}_t\rightarrow \widetilde{\Y}_t$, there exist the unique lifts of them and their lifts are contained in $\del K_+\cap \del K_-$.
\end{enumerate}

Then we can prove Proposition \ref{2chainprop}.

\begin{proof}[(Proof of Proposition \ref{2chainprop})]
It is enough to show that $K_+-K_-$ is a 2-chain associated with $(\xi_1)_t-(\xi_0)_t$ in a weak sense.
Note that $(\xi_1)_t-(\xi_0)_t$ is represented by the formal sum
\begin{equation*}
\left(\widetilde{\Cc}, \psi_0^{-1}\psi_1\right) + \left(Q_{(0,0)}, \varphi_0^{-1}\right) + \left(Q_{(1,1)},\varphi_1\right).
\end{equation*}
Let $\widetilde{\gamma}_c$ (resp. $\widetilde{\gamma}_{00}$, $\widetilde{\gamma}_{11}$) be the 1-chain on $\widetilde{\Cc}$ (resp. $Q_{(0,0)}$, $Q_{(1,1)}$)  defined by the pull-back of $[\infty,0]$ on $\P^1$ by the rational function $\psi_0^{-1}\psi_1$ (resp. $\varphi_0^{-1}$, $\varphi_1$). Then $\widetilde{\gamma}_c+\widetilde{\gamma}_{00}+\widetilde{\gamma}_{11}$ is the 1-cycle associated with $\xi$. By Figure \ref{Kfigure2}, we have
\begin{equation*}
\del (K_+ -K_-) = (\gamma_{c,+}-\gamma_{c,-}) + (\gamma_{00,+}-\gamma_{00,-}) +  (\gamma_{10,+} -\gamma_{10,-}) + (\gamma_{11,+}-\gamma_{11,-}).
\end{equation*}
Hence we should show that the 1-chains $(\gamma_{c,+}-\gamma_{c,-}-\widetilde{\gamma}_c)$, $(\gamma_{00,+}-\gamma_{00,-}-\widetilde{\gamma}_{00})$, $(\gamma_{10,+} -\gamma_{10,-})$ and $(\gamma_{11,+}-\gamma_{11,-}-\widetilde{\gamma}_{11})$ are 1-boundaries on the curves $\widetilde{\Cc},Q_{(0,0)},Q_{(1,0)}$ and $Q_{(1,1)}$ respectively.
Since $\widetilde{\Cc},Q_{(0,0)},Q_{(1,0)}$ and $Q_{(1,1)}$ are isomorphic to $\P^1$ and $H_1(\P^1)=0$, it is enough to show that they are 1-cycles. 
By Figure \ref{Kfigure2}, we have the following relations.
\begin{equation*}
\begin{aligned}
&\del(\gamma_{c,+}-\gamma_{c,-}) =\div_{\widetilde{\Cc}}(\psi_0^{-1}\psi_1) = \del\widetilde{\gamma}_c,&&\del(\gamma_{00,+}-\gamma_{00,-}) = \div_{Q_{(0,0)}}(\varphi_0^{-1}) = \del\widetilde{\gamma}_{00}\\
&\del(\gamma_{11,+}-\gamma_{11,-}) = \div_{Q_{(1,1)}}(\varphi_1) = \del\widetilde{\gamma}_{11},&& \del(\gamma_{10,+} -\gamma_{10,-})=0
\end{aligned}
\end{equation*}
Hence they are 1-cycles and we confirm that $K_+-K_-$ is a 2-chain associated with $(\xi_1)_t-(\xi_0)_t$ in a weak sense.
\end{proof}

Let $\Lc$ be the local holomorphic function
\begin{equation}\label{localLdef}
\Lc(t) = \int_{K_+}\omega_t-\int_{K_-}\omega_t
\end{equation}
which is defined around $t_0$.
Using the local description of $\omega_t$, we see that 
\begin{equation*}
\int_{K_+}\omega_t=-\int_{K_-}\omega_t = \int_{\triangle}\frac{dxdy}{\sqrt{x(1-x)(1-ax)}\sqrt{y(1-y)(1-by)}}.
\end{equation*}
Hence $\Lc(t)$ is a lift of $\Lc(a,b)$ in (\ref{functionL}).
Then we can prove Proposition \ref{preprop5}.

\begin{proof}[(Proof of Proposition \ref{preprop5})]
We will prove (1) and (2) simultaneously.
Let $\{U_i\}_i$ be a good open cover of $T$ such that for each $i$, there exists a local holomorphic function $\varphi_i$ on $U_i$ which represents $\nu_{\tr}(\xi_1-\xi_0)|_{U_i}$.
To prove (1) and (2), it is enough to show that $\Lc$ have an analytic continuation on each $U_i$ and the resulting function coincides with $\varphi_i$ up to an elements in $\Pc_\omega(U_i)$.

By Proposition \ref{2chainprop}, the equation (\ref{transregK}) holds for any point $t$ of an open neighborhood $U_0$ of $t_0$.
Since the left-hand side of (\ref{transregK}) corresponds to $\langle\ev_t(\nu_{\tr}(\xi_1-\xi_0)),[\omega_t]\rangle$, we see that $[\Lc]$ corresponds to $\nu_{\tr}(\xi_1-\xi_0)|_{U_0}$ by Corollary \ref{basiccor}.
Thus if $U_i\cap U_0\neq \emptyset$, we have $[\varphi_i]|_{U_i\cap U_0} = [\Lc]|_{U_i\cap U_0}$,
so $\Lc-\varphi_i \in \Pc_\omega(U_i\cap U_0)$.
Let $p_i$ be the restriction of $\Lc-\varphi_i$ on $U_i\cap U_0$.
Then we can extend $p_i$ on $U_i$.
The function $\varphi_i + p_i$ is a desired analytic continuation of $\Lc$ on $U_i$.
For a general $U_i$, take a finite family $U_{i_1},\dots, U_{i_n}=U_i$ such that $U_0\cap U_{i_1}\neq \emptyset$ and $U_{i_{k-1}}\cap U_{i_{k}}\neq \emptyset$ for $k = 2,\dots,n$ and repeat the above process.

Next, we will prove (3). 
We will compute $\Dc_1(\Lc)$.
It is enough to calculate it on a neighborhood of $t_0$. 
Hence we may assume $\Lc$ is given by (\ref{localLdef}).

Let $H(a,x)$ be a local holomorphic function defined by $H(a,x) =-\frac{\sqrt{x(1-x)}}{2\sqrt{1-ax}^3}$.
We can check that 
\begin{equation*}
\Dc_1 \left(\frac{1}{\sqrt{x(1-x)(1-ax)}}\right) = \frac{\del H(a,x)}{\del x}.
\end{equation*}
Then by Stokes' theorem\footnote{To be more precise, we should take a reparametrization of $K_+$ using $\gamma$ before Proposition \ref{2chainprop}.}, we have
\begin{equation*}
\begin{aligned}
&\Dc_1\left(\frac{1}{2}\Lc\right) = \Dc_1\left( \int_{K_+}\frac{dxdy}{\sqrt{x(1-x)(1-ax)}\sqrt{y(1-y)(1-by)}}\right) \\
&= \int_{K_+} d\left(\frac{H(a,x)dy}{\sqrt{y(1-y)(1-by)}}\right) = \int_{\del K_+} \frac{H(a,x)dy}{\sqrt{y(1-y)(1-by)}},
\end{aligned}
\end{equation*}
and since the 1-form $\frac{H(a,x)dy}{\sqrt{y(1-y)(1-by)}}$ vanishes on $\del K_+$ except $\gamma_{c,+}$, we have
\begin{equation*}
=\frac{1}{2}\int_0^1 \frac{dz}{(1-bz)^{\frac{1}{2}}(1-az)^{\frac{3}{2}}} =\frac{1}{a-b} \int_1^{\frac{\sqrt{1-b}}{\sqrt{1-a}}} du = \frac{1}{a-b}\cdot\left(\frac{\sqrt{1-b}}{\sqrt{1-a}}-1\right).
\end{equation*}
Here we use the coordinate transform $u= \frac{\sqrt{1-bz}}{\sqrt{1-az}}$. 
We can compute $\Dc_2(\Lc)$ similarly.
\end{proof}

\part{The group action}
\section{The group action on the Kummer family}
In this section, we construct a group action on $\widetilde{\X}\rightarrow T$. 
To produce many higher Chow cycles, it is not enough to consider only automorphisms which are trivial on $T$.
We will consider automorphisms of the following type.
\begin{defn}\label{autoasfamily}
Let $X\rightarrow S$ be a family of algebraic varieties. The {\it automorphism group $\Aut(X\rightarrow S)$ of $X\rightarrow S$} consists of a pair $(g,h)$ with $g\in \Aut(X)$ and $h\in \Aut(S)$ such that the following diagram commutes.
\begin{equation*}
\begin{tikzcd}
X \arrow[r,"g"]\arrow[d] & X \arrow[d] \\
S \arrow[r,"h"]          & S
\end{tikzcd}
\end{equation*}
We say that a group $G$ acts on a family of varieties $X\rightarrow S$ if we have a group homomorphism $G\rightarrow \Aut(X\rightarrow S)$.
\end{defn}
The main result of this section is as follows.
\begin{prop}\label{mainautoprop}
There exists a $G_{\X}$-action on $\widetilde{\X}\rightarrow T$ where $G_{\X}$ is a $\Z/2$-extension of $(\mathfrak{S}_4\times_{\mathfrak{S}_3} \mathfrak{S}_4)^2$.
\end{prop}
We use the following conventions for group actions.
\begin{enumerate}
\item In this paper, we consider \textit{left} group actions unless specified otherwise.
\item Suppose a group $G$ acts on a set $M$. The group $G$ \textit{stabilizes} a subset $M'$ of $M$ if we have $g\cdot m\in M'$ for any $g\in G$ and $m\in M'$.
\item For a $\C$-scheme $X$, $\Aut(X)$ denotes the $\C$-automorphism group of $X$.
For $\rho\in \Aut(X)$, $\rho^\sharp$ denotes the morphism of sheaves $\O_X\rightarrow \rho_*\O_X$.
\item For $n\in \Z_{>1}$, $\mu_n$ denotes the group of $n$-th roots of unity.
\item For $n\in \Z_{\ge 1}$, $\mathfrak{S}_n$ denotes the symmetric group of degree $n$.
For a set $M$, $\mathfrak{S}(M)$ denotes the symmetric group of $M$.
The sign character is denoted by $\sgn: \mathfrak{S}(M)\rightarrow \mu_2$.
\end{enumerate}
Before the construction of the $G_{\X}$-action on $\widetilde{\X}\rightarrow T$, we recall some formal properties of finite group actions on schemes in Section 6.1. 
\subsection{Generalities}

\textit{A scheme with a group action} $(S,H,\varphi)$ is a triplet consisting of a $\C$-scheme $S$, a group $H$ and a group homomorphism $\varphi: H\rightarrow \Aut(S)$.
We usually omit $\varphi$ from the notation and write $(S,H)$. 
A morphism $(X,G)\rightarrow (S,H)$ between two such objects is a pair of a morphism $X\rightarrow S$ of varieties and a group homomorphism $G\rightarrow H$ satisfying the usual compatibility condition.

In Section 6.2, we use the following fiber product construction.
Suppose we have two morphisms $(\pi,\psi): (X,G)\rightarrow (S,H)$ and $(f,\varphi):(S',H')\rightarrow (S,H)$. 
Recall that the fiber product of $G$ and $H'$ over $H$ is defined by 
\begin{equation*}
G\times_{H} H' = \{(g,h)\in G\times H': \psi(g) = \varphi(h)\}.
\end{equation*}
Then we define the fiber product of $(X,G)$ and $(S',H')$ over $(S,H)$ by 
\begin{equation}\label{fiberprodprop}
(X\times_{S} S', G\times_{H}H').
\end{equation}
This is indeed the fiber product in the category of schemes with group actions.

Let $(X,G)$ be a scheme with a group action. 
Then we have the natural right $G$-action on $\Gamma(X,\O_X^\times)$.
A \textit{1-cocycle} on $\Gamma(X,\O_X^\times)$ is a map $\chi: G\rightarrow \Gamma(X,\O_X^\times)$ satisfying
\begin{equation*}
\chi(gh) = h^\sharp(\chi(g))\cdot \chi(h)
\end{equation*}
for any $g,h\in G$. Here $h^\sharp: \O_X\rightarrow h_*\O_{X}$ is the natural morphism induced by $h:X\rightarrow X$.

Let $\L$ be a $G$-linearized line bundle on $X$. The $G$-linearization is the same as a collection 
$(\Phi_{g}: g^*\L\xrightarrow{\sim} \L)_{g\in G}$ of isomorphisms satisfying the relation $\Phi_{h}\circ h^*(\Phi_g) = \Phi_{gh}$ for $g,h\in G$.
We can construct a 1-cocycle as follows. 
Let $s$ be a global section of $\L$ such that $\div(s)$ is stable under the $G$-action.
Then there exists the unique 1-cocycle $\chi:G\rightarrow \Gamma(X,\O_X^\times)$ such that  
\begin{equation*}
\Phi_g(g^*(s)) = \chi(g)^{-1}\cdot s\quad (g\in G).
\end{equation*}

Finally, we prove liftability of a group action by a cyclic covers.
For a line bundle $\L$ and a section $s\in \Gamma(X,\L^{\otimes (-m)})$, the cyclic covering $Y\rightarrow X$ associated with $(\L,s)$ is defined as the relative spectrum 
of the $\O_X$-algebra $\bigoplus_{i=0}^{m-1}\L^{\otimes i}$
(cf. \cite{BPHV} p.~54).
The branching locus of $Y\rightarrow X$ coincides with $\div(s)$.
\begin{prop}\label{liftprop}
Let $\pi:Y\rightarrow X$ be the degree $m$ cyclic covering associated with  $(\L,s)$.
Suppose a group $G$ acts on $X$ and a $G$-linearization $(\Phi_g)_{g\in G}$ of $\L$ are given.
Assume that there exists a 1-cocycle $\chi:G\rightarrow \Gamma(X,\O_X^\times)$ such that 
\begin{equation*}
\Phi_g^{\otimes (-m)}(g^*(s)) = \chi(g)^{-m}\cdot s
\end{equation*}
for any $g\in G$.
Then there exists a $G$-action on $Y$ such that $\pi$ is $G$-equivariant.
\end{prop}
\begin{proof}
For $g\in G$, we define an automorphism $\tilde{g}: Y\rightarrow Y$ as follows. 
\begin{enumerate}
\item Let $Y_1$ be the cyclic covering associated with $(g^*\L, g^*(s))$. Then $Y_1$ is the fiber product of $Y\rightarrow X$ and $X\xrightarrow{\:g\:} X$. Since $g$ is an isomorphism, $Y_1\rightarrow Y$ is an isomorphism as well.
\item Let $Y_2$ be the cyclic covering associated with $(\L,\chi(g)^{-m}\cdot s)$. Then we have the isomorphism $Y_2\simeq Y_1$ over $X$ induced by $\Phi_g$.
\item We have the $\O_X$-module automorphism on $\L$ defined by the multiplication of $\chi(g)^{-1}$.
This automorphism induces an isomorphism $Y\simeq Y_2$ over $X$.
\end{enumerate}
Composing these isomorphisms, we get an automorphism $\tilde{g}\in \Aut(Y)$.
\begin{equation}
\begin{tikzcd}
Y \arrow[r,"\mathrm{(3)}","\sim"'] \arrow[d,"\pi"]& Y_2 \arrow[r,"\mathrm{(2)}","\sim"']\arrow[d,"\pi"] & Y_1 \arrow[d,"\pi"] \arrow[r,"\mathrm{(1)}","\sim"'] &Y \arrow[d,"\pi"] \\
X \arrow[r,equal]&X\arrow[r,equal]& X \arrow[r,"g","\sim"']& X
\end{tikzcd}
\end{equation}
We can show that $G\rightarrow \Aut(Y);g\mapsto \tilde{g}$ is a group homomorphism by the condition on $(\Phi_{g})_{g\in G}$ and the property of the 1-cocycles.
Hence we can construct a $G$-action on $Y$ and $\pi$ is $G$-equivariant by the construction.
\end{proof}

\subsection{Construction of the groups and their actions}
In this section, we will construct a group action on $\widetilde{\X}\rightarrow T$ by 4 steps. 
\begin{enumerate}
\renewcommand{\labelenumi}{\theenumi. }
\item We construct a $\mathfrak{S}(\Sigma)$-action on $\P^1\times S_0\rightarrow S_0$.
By taking its direct product, we have a $G_{\Y_0}(=\mathfrak{S}(\Sigma)^2)$-action on $\Y_0\rightarrow T_0$ which stabilizes $\Sigma^2$.
\item By taking the fiber product of the $G_{\Y_0}$-action on $\Y_0\rightarrow T_0$ and a $G_{T}(=\mathfrak{S}_4^2)$-action on $T\rightarrow T_0$, we have a $G_{\Y}$-action on $\Y\rightarrow T$.
\item By considering a $\mu_2$-extension $G_{\X}$ of $G_{\Y}$, we can construct a $G_{\X}$-action on $\X \rightarrow T$ by applying Proposition \ref{liftprop}.
\item We lift the $G_{\X}$-action on $\X\rightarrow T$ to $\widetilde{\X}\rightarrow T$.
\end{enumerate}

Throughout in this section, we identify $\Aut(S_0) = \mathfrak{S}(\{0,1,\infty\}) (\simeq \mathfrak{S}_3)$.

\begin{prop}\label{psprop}
There exists a $\mathfrak{S}_4$-action on $\P^1\times S_0\rightarrow S_0$ satisfying the following properties.
\begin{enumerate}
\item The natural map $\mathfrak{S}_4\rightarrow \Aut(S_0)=\mathfrak{S}(\{0,1,\infty\})$ is surjective. For $\rho\in \mathfrak{S}_4$, its image in $\mathfrak{S}(\{0,1,\infty\})$ is denoted by $\underline{\rho}$.
\item For $\rho\in \mathfrak{S}_4$ and $\sigma \in \Sigma$, we have $\rho\circ \sigma \circ \underline{\rho}^{-1}\in \Sigma$. Moreover, the following map induces an isomorphism of groups.
\begin{equation}\label{S4isom}
\begin{tikzcd}[row sep = tiny]
\mathfrak{S}_4 \arrow[r] & \mathfrak{S}(\Sigma) \arrow[r,equal] &[-15pt] \mathfrak{S}(\{0,1,1/c,\infty\}) \\
\rho \aru \arrow[rr,mapsto] & & \left(\sigma\mapsto \rho \circ \sigma \circ \underline{\rho}^{-1}\right)\aru
\end{tikzcd}
\end{equation}
\end{enumerate}
Since another group which is isomorphic to $\mathfrak{S}_4$ appears elesewhere, hereafter we denote this group by $\mathfrak{S}(\Sigma)$.
\end{prop}
\begin{proof}
Consider the following contravariant functor from the category of varieties to the category of sets.
\begin{equation*}
\begin{tikzcd}[ampersand replacement=\&]
X\arrow[r,mapsto] \&[-10pt]\left\{(Y\rightarrow X;p_1,p_2,p_3,p_4): 
\mathrm{
\begin{aligned}
& Y\rightarrow X \text{ is a $\P^1$-bundle.}\\
& p_1,\cdots, p_4\text{ are its sections such that}\\
& p_i(X)\cap p_j(X)=\emptyset  \text{ for any }i\neq j. \\
\end{aligned}} \right\}/(\text{$X$-isom.})
\end{tikzcd}
\end{equation*}
For each $(Y\rightarrow X;p_1,p_2,p_3,p_4)$, by considering the cross-ratio of $p_1(x)$, $p_2(x)$, $p_3(x)$, $p_4(x)\in \P^1$ at $x\in X$, we have the morphism $X\rightarrow S_0$.
Then we can check that $S_0$ represents this functor and $(\P^1\times S_0\rightarrow S_0;0,1,\infty,1/c)$ is the universal element. 
For each $\rho\in \mathfrak{S}_4$, $\rho$ induces a natural automorphism on this functor. 
Combining with the fact that an automorphism on $\P^1$ which stabilizes 4 points is identity, this induces a $\mathfrak{S}_4$-action on $\P^1\times S_0\rightarrow S_0$. 
The property (2) is clear from this construction. 
The property (1) follows from the property of the cross-ratio.
\end{proof}
We need an explicit description of the $\mathfrak{S}(\Sigma)$-action on $\P^1\times S_0$ for Step 3. We summarize them on Table \ref{psgroupisomtable}. In the table, for each $\rho\in \mathfrak{S}(\Sigma)$, the image of $c$ under $\underline{\rho}^\sharp :\O_{S_0}\rightarrow \underline{\rho}_*\O_{S_0}$ and the image of the local coordinate $z$ on $\P^1$ under $\rho^\sharp: \O_{\P^1\times S_0}\rightarrow \rho_*\O_{\P^1\times S_0}$ are given.
{\small 
\begin{table}[h]
\caption{The $\mathfrak{S}(\Sigma)$-action on $\P^1\times S_0$}
\label{psgroupisomtable}
{\renewcommand\arraystretch{1.3}
\begin{tabular}{c|c|c}
$\rho$ & $\underline{\rho}^\sharp(c)$ & $\rho^\sharp(z)$ \\ \hline
$\id$ & $c$& $z$  \\ 
\end{tabular}
\vspace{10pt}\\
\begin{tabular}{c|c|c||c|c|c||c|c|c}
$\rho$ & $\underline{\rho}^\sharp(c)$ & $\rho^\sharp(z)$                               &$\rho$ & $\underline{\rho}^\sharp(c)$ & $\rho^\sharp(z)$                                      &$\rho$ & $\underline{\rho}^\sharp(c)$ & $\rho^\sharp(z)$  \\ 
\hline
$(0\; 1)  $        &$\frac{c}{c-1}$           &$1-z$              &$(0\;1/c)   $         & $1-c$& $\frac{1-c z}{1-c}$ &$(0\;\infty) $  &$\frac{1}{c}$  & $\frac{1}{z}$ \\ 
$(1/c\;\infty)$  &$\frac{c}{c-1}   $             & $\frac{(1-c)z}{1-c z}$  &$(1\;\infty) $  &$ 1-c $ & $\frac{z}{z-1}$&$(1\; 1/c) $          &$\frac{1}{c}$&$c z$ \\
\end{tabular}
\vspace{10pt}\\
\begin{tabular}{c|c|c||c|c|c||c|c|c}
$\rho$ & $\underline{\rho}^\sharp(c)$ & $\rho^\sharp(z)$   &$\rho$ & $\underline{\rho}^\sharp(c)$ & $\rho^\sharp(z)$
&$\rho$ & $\underline{\rho}^\sharp(c)$ & $\rho^\sharp(z)$                          \\ 
\hline
$(0\;1)(1/c\;\infty)$& $c$ &$\frac{1-z}{1-c z}$&$(0\;1/c)(1\;\infty)$&$c$ & $\frac{1-c z}{c (1-z)} $ &$(0\;\infty)(1\;1/c) $& $c$& $\frac{1}{c z}$\\
\end{tabular}
\vspace{10pt}\\
\begin{tabular}{c|c|c||c|c|c}
 $\rho$ & $\underline{\rho}^\sharp(c)$ & $\rho^\sharp(z)$
 &$\rho$ & $\underline{\rho}^\sharp(c)$ & $\rho^\sharp(z)$\\
\hline
$(0\;1\;1/c)    $    & $\frac{1}{1-c}$ &$1-c z$ &$(0\;1/c \;1)     $    & $\frac{c-1}{c}$&  $\frac{c(1-z)}{c -1}$\\

$(0\;\infty\;1)$& $\frac{1}{1-c} $&$\frac{z-1}{z}$&$(0\;1\;\infty)$& $\frac{c-1}{c}$ & $\frac{1}{1-z}$ \\

$(0\;1/c\;\infty)$&$ \frac{1}{1-c}$ &$\frac{1-c}{1-c z}$&$(0\;\infty\;1/c)$&$\frac{c}{c-1}$& $ \frac{1-c z}{(1-c)z}$\\

$(1\;\infty\;1/c)$&  $\frac{1}{1-c}$& $\frac{(c-1)z}{1-z}$&$(1\;1/c\;\infty)$&$\frac{c-1}{c}$&$\frac{c z}{c z-1}$\\

\end{tabular}
\vspace{10pt}\\
\begin{tabular}{c|c|c||c|c|c||c|c|c}
$\rho$ & $\underline{\rho}^\sharp(c)$ & $\rho^\sharp(z)$
& $\rho$ & $\underline{\rho}^\sharp(c)$ & $\rho^\sharp(z)$
& $\rho$ & $\underline{\rho}^\sharp(c)$ & $\rho^\sharp(z)$ \\                                                                                      
\hline
$(0\;1/c\;1\;\infty)$ &$\frac{c}{c-1} $&$\frac{c-1}{c(1-z)}$&$(0\;1\;1/c\;\infty) $&$ 1-c$ &$\frac{1}{1-c z}$&$  (0\;1\;\infty\;1/c)$&$\frac{1}{c}  $  & $\frac{1-c z}{1-z} $\\
$(0\;\infty\;1\;1/c)$ &$\frac{c}{c -1}$ &$\frac{c z-1}{c z}$ &$(0\;\infty\;1/c\;1)$& $1-c$&$\frac{1-z}{(c-1)z}$ &$(0\;1/c\;\infty\;1)$& $\frac{1}{c}$    &$ \frac{c(1-z)}{1-c z}$     \\
\end{tabular}
}
\end{table}
}

From Proposition \ref{psprop}, we have a morphism $(\P^1\times S_0,\mathfrak{S}(\Sigma))\rightarrow (S_0,\mathfrak{S}(\{0,1,\infty\}))$. By taking the direct product, we have the morphism 
\begin{equation}
\label{dirprodauto}
\begin{tikzcd} 
(\P^1\times S_0\times \P^1\times S_0, \mathfrak{S}(\Sigma)\times \mathfrak{S}(\Sigma)) \arrow[r] & (S_0\times S_0, \mathfrak{S}(\{0,1,\infty\}) \times \mathfrak{S}(\{0,1,\infty\}) ).
\end{tikzcd}
\end{equation}
We denote the group $\mathfrak{S}(\Sigma)\times \mathfrak{S}(\Sigma)$ (resp. $\mathfrak{S}(\{0,1,\infty\}) \times \mathfrak{S}(\{0,1,\infty\})$) by $G_{\Y_0}$ (resp. $G_{T_0}$). Since $T_0\subset S_0\times S_0$ is stable under the $G_{T_0}$-action, (\ref{dirprodauto}) induces a  morphism $(\Y_0,G_{\Y_0})\rightarrow (T_0,G_{T_0})$. 
Furthermore, by the property (2) of Proposition \ref{psprop}, we see that $\Sigma^2\subset \Y_0$ is stable under the $G_{\Y_0}$-action on $\Y_0$.
We finish Step 1.

Next, we will construct a group action on $T$.
Recall that $T$ is a finite \'etale cover of $T_0$ corresponding to $B_0 \rightarrow B_0[\sqrt{a},\sqrt{b},\sqrt{1-a},\sqrt{1-b}]$.

\begin{prop}\label{autosp}
Let $S'=\Spec A_0[\sqrt{c},\sqrt{1-c}]$. Then there exists a $\mathfrak{S}_4$-action on $S'\rightarrow S_0$ such that the natural map $\mathfrak{S}_4\rightarrow \Aut(S_0)=\mathfrak{S}(\{0,1,\infty\})$ is surjective.
\end{prop}
\begin{proof}
We have the following ring isomorphism.
\begin{equation}
\label{ringisom}
\begin{tikzcd}
A_0\left[\sqrt{c},\sqrt{1-c}\right] \arrow[r,"\sim"] &[-10pt] \C\left[\gamma,\frac{1}{\gamma(\gamma^4-1)}\right]; \sqrt{c},\sqrt{1-c} \arrow[r,mapsto] & \frac{\gamma+\frac{1}{\gamma}}{2}, \frac{\gamma-\frac{1}{\gamma}}{2\sqrt{-1}}
\end{tikzcd}
\end{equation}
Using the ring isomorphism (\ref{ringisom}), we can check directly that for each $\rho \in \Aut(S_0)$, there exists a lift of $\rho$ in $\Aut(S')$.
Moreover, the $S_0$-automorphism group of $S'$ is isomorphic to $\mu_2\times \mu_2$. Thus we have the following exact sequence of groups. 
\begin{equation*}
1 \lra \mu_2\times \mu_2 \lra \Aut(S'\rightarrow S_0) \lra \Aut(S_0) \lra 1
\end{equation*}
To prove the proposition, we should show that $\Aut(S'\rightarrow S_0)$ is isomorphic to $\mathfrak{S}_4$. We will construct an $\Aut(S'\rightarrow S_0)$-action on a set of cardinality $4$. By the ring isomorphism (\ref{ringisom}), we have $S' \simeq \P^1-\{0,\infty,\pm 1, \pm\sqrt{-1}\}$. Thus $\Aut(S'\rightarrow S_0)$ acts on the set $\{0,\infty,\pm 1, \pm \sqrt{-1}\}$. Let $F$ be the set
\begin{equation*}
F = \left\{
\begin{aligned}
&\{\{1,\sqrt{-1},\infty\},\{-1,-\sqrt{-1},0\}\}, && \{\{-1,\sqrt{-1},\infty\},\{1,-\sqrt{-1},0\}\} \\
&\{\{1,-\sqrt{-1},\infty\},\{-1,\sqrt{-1},0\}\}, &&\{\{-1,-\sqrt{-1},\infty\},\{1,\sqrt{-1},0\}\} 
\end{aligned}
\right\}.
\end{equation*}
If we plot $\pm 1, \pm\sqrt{-1},0,\infty$ on the Riemann sphere, they are the vertexes of the regular octahedron. Then $F$ is the set of pairs of opposite faces of the octahedron. We can check that $\Aut(S'\rightarrow S_0)$ acts on $F$ and the group homomorphism $\Aut(S'\rightarrow S_0)\rightarrow \mathfrak{S}(F)$ is surjective. By the exact sequence above, the orders of $\Aut(S'\rightarrow S_0)$ and $\mathfrak{S}(F)$ coincide. Thus $\Aut(S'\rightarrow S_0)\simeq \mathfrak{S}(F)$.
\end{proof}
From Proposition \ref{autosp}, we have a morphism $(S',\mathfrak{S}_4)\rightarrow (S_0,\mathfrak{S}(\{0,1,\infty\}))$.
By taking its direct product, we have
\begin{equation}\label{dirprodSS}
\begin{tikzcd}
(S'\times S',\mathfrak{S}_4\times \mathfrak{S}_4)\arrow[r] & (S_0\times S_0,G_{T_0}).
\end{tikzcd}
\end{equation}
We denote the group $\mathfrak{S}_4\times \mathfrak{S}_4$ by $G_{T}$.
Since $T$ is a fiber product of $T_0$ and $S'\times S'$ over $S_0\times S_0$, this induces a morphism $(T,G_{T})\rightarrow (T_0,G_{T_0})$.

We have already constructed the morphisms $(\Y_0,G_{\Y_0})\rightarrow (T_0,G_{T_0})$ and $(T,G_{T})\rightarrow (T_0,G_{T_0})$.
Then the group
\begin{equation*}
G_{\Y_0}\times_{G_{T_0}} G_{T}  = (\mathfrak{S}(\Sigma)^2)\times_{\mathfrak{S}(\{0,1,\infty\})^2} (\mathfrak{S}_4^2)\simeq (\mathfrak{S}_4\times_{\mathfrak{S}_3}\mathfrak{S}_4)^2
\end{equation*}
acts on $\Y = \Y_0\times_{T_0} T$ by (\ref{fiberprodprop}). 
We denote this group by $G_\Y$.
Thus we have the morphism $(\Y,G_\Y)\rightarrow (T,G_{T})$.
We often denote an element of $G_\Y$ by a 3-tuple
\begin{equation*}
(\rho_1,\rho_2,\tau) \in \mathfrak{S}(\Sigma)\times \mathfrak{S}(\Sigma)\times G_{T}
\end{equation*}
such that $\tau\in G_{T}$ and $(\rho_1,\rho_2)\in G_{\Y_0}$ have the same image in $G_{T_0}=\mathfrak{S}(\{0,1,\infty\})^2$.
We finish Step 2.

We will lift this group action to the double cover $\X\rightarrow \Y$. Let $\L$ be the line bundle on $\Y$ defined by 
\begin{equation*}
\L = pr_1^*\Omega^1_{\P^1\times S_0/S_0} \otimes pr_2^*\Omega^1_{\P^1\times S_0/S_0} 
\end{equation*}
where $pr_i$ is a morphism $\Y\rightarrow \Y_0 \hookrightarrow (\P^1\times S_0)\times (\P^1\times S_0) \xrightarrow{pr_i} \P^1\times S_0$. Let $s$ be the section of $\L^{\otimes(-2)}$ defined by 
\begin{equation*}
s = f(x)g(y)(dx\otimes dy)^{\otimes(-2)}.
\end{equation*}
We have the natural $G_{\Y}$-linearization $(\Phi_\rho)_{\rho\in G_{\Y}}$ of $\L$ since it is constructed from sheaves of differentials.

The morphism $\X\rightarrow \Y$ is the double covering associated with $(\L,s)$.
Since $\div(s)$ is equal to the branching locus of $\X\rightarrow \Y$, $\div(s)$ is stable under the $G_{\Y}$-action.
Hence we have the 1-cocycle $\eta: G_{\Y}\rightarrow \Gamma(\Y,\O_{\Y}^\times) = B^\times$ such that 
\begin{equation*}
\Phi_\rho^{\otimes(-2)}(\rho^*(s)) = \eta(\rho)^{-1}\cdot s
\end{equation*}
for any $\rho\in G_{\Y}$.
We can compute $\eta$ as follows.

\begin{prop}\label{chikeisan}
For $\rho = (\rho_1,\rho_2,\tau)\in G_{\Y}$, $\eta(\rho)$ is determined by $\underline{\rho}_1, \underline{\rho}_2 \in \mathfrak{S}(\{0,1,\infty\})$ and calculated by the following formula.
\begin{equation*}
\eta(\rho) = \eta_1(\underline{\rho}_1)\cdot \eta_2(\underline{\rho}_2)
\end{equation*}
where $\eta_1(\underline{\rho}_1), \eta_2(\underline{\rho}_2)\in B^\times$ are given by the following table.
\begin{table}[H]
{\renewcommand\arraystretch{1.3}
\begin{tabular}{c|c||c|ccc|c||c|c}
$\underline{\rho}_1$ & $\eta_1(\underline{\rho}_1)$  & $\underline{\rho}_1$ &$\eta_1(\underline{\rho}_1)$  &\phantom{hogeho} & $\underline{\rho}_2$ & $\eta_2(\underline{\rho}_2)$  & $\underline{\rho}_2$& $\eta_2(\underline{\rho}_2)$ \\ \cline{1-4}\cline{6-9}
$\id$ & $ 1 $ & $(0\:1)$ &  $-1$  & & $\id$ & $ 1 $ & $(0\:1)$ &  $-1$ \\
$(1\:\infty)$ & $1-a$ &$(0\:1\:\infty)$ &$a-1$  && $(1\:\infty)$ & $1-b$ &$(0\:1\:\infty)$ &$b-1$ \\
$(0\:\infty)$ & $a$ & $(0\:\infty\:1)$& $ -a$ && $(0\:\infty)$ & $b$ & $(0\:\infty\:1)$& $ -b$\\
\end{tabular}
}
\end{table}
\end{prop}
\begin{proof}
By the definition of the $G_{\Y}$-linearization on $\L$, we have 
\begin{equation*}
\Phi_{\rho}(\rho^*(dx\otimes dy)) = \left(\frac{\del }{\del x}\rho_1^\sharp(x)\right)\left(\frac{\del }{\del y}\rho_2^\sharp(y)\right)\cdot (dx\otimes dy).
\end{equation*}
Using Table \ref{psgroupisomtable}, we can compute $\eta$ for each $\rho_1,\rho_2\in \mathfrak{S}(\Sigma)$ and we get the result.
\end{proof}
To apply Proposition \ref{liftprop}, we should construct a ``square root" of the 1-cocycle $\eta$. We define coboundary 1-cocycles $\phi_1,\phi_2:G_{T}\rightarrow B^\times$ by 
\begin{equation*}
\phi_1(\tau) = \tau^\sharp\left(\frac{\sqrt{a}\sqrt{1-a}}{a^2-a+1}\right)\cdot \left(\frac{\sqrt{a}\sqrt{1-a}}{a^2-a+1}\right)^{-1},\:\phi_2(\tau) = \tau^\sharp\left(\frac{\sqrt{b}\sqrt{1-b}}{b^2-b+1}\right)\cdot \left(\frac{\sqrt{b}\sqrt{1-b}}{b^2-b+1}\right)^{-1}
\end{equation*}
for $\tau\in G_{T}$. 
Then we can check that for any $\rho=(\rho_1,\rho_2,\tau)\in G_{\Y}$, 
\begin{equation}\label{sqrteqn}
\eta_i(\underline{\rho}_i) = \sgn(\underline{\rho}_i)\phi_i(\tau)^2
\end{equation}
holds for $i=1,2$.
Hence $\phi_1\cdot\phi_2$ coincides with a square root of $\eta$ up to sign.
To get a square root of sign, we enlarge the group $G_{\Y}$ as follows.

We define the group $G_{\X}$ as the fiber product $G_{\Y}\times_{\mu_2} \mu_4$ where the group homomorphisms $G_{\Y}\rightarrow \mu_2$ and $\mu_4\rightarrow \mu_2$ is given by
\begin{equation*}
\begin{aligned}
&G_{\Y} \rightarrow \mu_2; && (\rho_1,\rho_2,\tau) \mapsto \sgn(\underline{\rho}_1)\sgn(\underline{\rho}_2) \\
&\mu_4 \rightarrow \mu_2; && \zeta \mapsto \zeta^2.
\end{aligned}
\end{equation*}
We denote an element of $G_{\X}$ by a 4-tuple 
\begin{equation*}
(\rho_1,\rho_2,\tau,\zeta) \in \mathfrak{S}(\Sigma)\times \mathfrak{S}(\Sigma)\times G_{T}\times \mu_4.
\end{equation*}
The group $G_{\X}$ acts on $\Y$ through the surjection $G_{\X}\twoheadrightarrow G_{\Y}$.
For $\rho = (\rho_1,\rho_2,\tau,\zeta)\in G_{\X}$, we define 
\begin{equation*}
\chi(\rho) = \zeta\cdot \phi_1(\tau)\cdot \phi_2(\tau) \quad \in B^\times=\Gamma(\Y,\O_\Y^\times).
\end{equation*}
Then by (\ref{sqrteqn}), $\chi:G_{\X}\rightarrow \Gamma(\Y,\O_{\Y}^\times)$ is a 1-cocycle satisfying the relation
\begin{equation*}
\chi(\rho)^2 = \eta(\rho).
\end{equation*}
Hence we have the following result.
\begin{prop}
There exists a $G_{\X}$-action on $\X$ such that the morphism $\X\rightarrow \Y$ is $G_{\X}$-equivariant.
\end{prop}

\begin{proof}
The $G_{\Y}$-linearization $(\Phi_\rho)_{\rho \in G_{\Y}}$ of $\L$ induces the natural $G_{\X}$-linearization of $\L$.
This linearization satisfies the equation
\begin{equation*}
\Phi_{\rho}^{\otimes(-2)}(\rho^*(s)) = \eta(\rho)^{-1}\cdot s = \chi(\rho)^{-2}\cdot s.
\end{equation*}
Hence by applying Proposition \ref{liftprop}, we have the result.
\end{proof}
Thus we have the morphism $(\X,G_{\X})\rightarrow (T,G_{T})$. We finish Step 3.

Recall that the set $\Sigma^2$ on $\Y_0$ is stable under the $G_{\Y_0}$-action. 
Then its base change $\Sigma^2$ on $\Y$ is also stable under the $G_{\Y}$-action. 
Since $\Sigma^2\subset \Y$ is contained in the branching locus of $\X\rightarrow \Y$ and $\X\rightarrow \Y$ is $G_{\X}$-equivariant, $\Sigma^2$ on $\X$ is also $G_{\X}$-stable. 
Then by the universality of the blowing-up, for any $\rho\in G_{\X}$, there exists the unique lift $\widetilde{\rho}\in \Aut(\widetilde{\X})$ of $\rho$.
By the uniqueness, $G_{\X}\rightarrow \Aut(\widetilde{\X});\rho\mapsto \widetilde{\rho}$ defines a $G_{\X}$-action on $\widetilde{\X}$. We finish Step 4.

Finally, we can prove the main result of this section.
\begin{proof}[(Proof of Proposition \ref{mainautoprop})] We have the morphisms
\begin{equation*}
(\widetilde{\X},G_{\X})\lra  (\X,G_{\X}) \lra  (\Y,G_{\Y}) \lra (T,G_{T}).
\end{equation*}
By composing these morphisms, we have the $G_{\X}$-action on the family $\widetilde{\X}\rightarrow T$.
\end{proof}

At the end of this section, we see an explicit description of the $G_{\X}$-action on $\widetilde{\X}$ using the local coordinates $x,y,v$ on $V$.
Let $\rho=(\rho_1,\rho_2,\tau,\zeta)\in G_{\X}$.
Since $\widetilde{\X}\rightarrow \Y$ is $G_{\X}$-equivariant, we see that 
\begin{equation*}
x\mapsto  \rho_1^\sharp(x),\quad y\mapsto  \rho_2^\sharp(y).
\end{equation*}
Locally on $U\subset \Y$, the morphism $\X\rightarrow \Y$ is described by the ring homomorphism
\begin{equation*}
\begin{tikzcd}
B[x,y]\arrow[r] & B[x,y,u]/(u^2-f(x)g(y))
\end{tikzcd}
\end{equation*}
where $u$ corresponds to the local section $(dx\otimes dy)$ of $\L$.
By the construction of the $G_{\X}$-action, the local coordinate $u$ transforms
\begin{equation*}
u \mapsto \chi(\rho)^{-1}\frac{\del}{\del x}(\rho_1^\sharp(x))\frac{\del}{\del y}(\rho_2^\sharp(y)) u.
\end{equation*}
Since $u=vf(x)$, by the formula $\rho_1^\sharp(f(x))=\eta_1(\underline{\rho}_1)^{-1}\left(\frac{\del}{\del x}\rho_1^\sharp(x)\right)^2f(x)$, $v$ transforms as follows:
\begin{equation*}
v\mapsto \frac{\eta_1(\underline{\rho}_1)}{\chi(\rho)}\frac{\frac{\del}{\del y}(\rho_2^\sharp(y))}{\frac{\del}{\del x}(\rho_1^\sharp(x))}v.
\end{equation*}

\section{The group action on the higher Chow cycles}
In this section, we prove the main result in this paper by using some technical propositions which is proved Section 8.
\subsection{The group action on the higher Chow cycles}
First, we produce more families of higher Chow cycles from $\xi_0,\xi_1$ and $\xi_\infty$ by using the group action.

For $t\in T$, $\rho=(\rho_1,\rho_2,\tau,\zeta)\in G_{\X}$ induces the isomorphism $\rho_t:\widetilde{\X}_t\rightarrow \widetilde{\X}_{\tau(t)}$ such that 
\begin{equation}\label{rhot}
\begin{tikzcd}
\widetilde{\X}_t \arrow[d,hook'] \arrow[r,"\rho_t",dashrightarrow] & \widetilde{\X}_{\tau(t)}  \arrow[d,hook'] \\
\widetilde{\X} \arrow[r,"\rho"] & \widetilde{\X}
\end{tikzcd}
\end{equation}
commutes. 
The morphism $\rho_t$ induces a map $\rho_t^*:\CH^2(\widetilde{\X}_{\tau(t)},1)\rightarrow \CH^2(\widetilde{\X}_t,1)$.
For an algebraic family of higher Chow cycles $\xi=\{\xi_t\}_{t\in T}$ on $\widetilde{\X}\rightarrow T$, we define the algebraic family $\rho^*\xi$ of higher Chow cycles by $\rho^*\xi = \{\rho_t^*(\xi_{\tau(t)})\}_{t\in T}$. 

In particular, we have defined families of higher Chow cycles $\rho^*\xi_0, \rho^*\xi_1$ and $\rho^*\xi_\infty$ for $\rho\in G_{\X}$.
For each $t\in T$, we define the subgroup $\Xi_t$ of $\CH^2(\widetilde{\X}_t,1)_{\ind}$ by 
\begin{equation*}
\Xi_t = \left\langle ((\rho^*\xi_\bullet)_t)_\ind  : \bullet \in\{ 0,1,\infty\}, \rho\in G_{\X}\right\rangle_{\Z}.
\end{equation*}

Then we can state the main theorem.

\begin{thm}\label{mainthm}
For a very general $t\in T$, we have $\rank \Xi_t \ge 18$.
In particular, $\rank \CH^2(\X_t,1)_\ind \ge 18$.
\end{thm}

To be more specific, 18 linearly independent cycles in $\Xi_t$ are given as follows.
Let $\rho^1,\rho^2,\dots,\rho^6\in G_{\X}$ be lifts of 
$(\id,\id)$, $(\id,(0\:\:1))$, $(\id,(1\:\:\infty))$, $(\id,(0\:\:1\:\:\infty))$, $(\id, (0\:\:\infty))$ and $(\id, (0\:\:\infty\:\:1))\in G_{\Y_0}$ by $G_{\X}\twoheadrightarrow G_{\Y_0}=\mathfrak{S}(\Sigma)^2$, respectively.
Then we will show that $(((\rho^i)^*\xi_{\bullet})_t)_\ind \:\:(i=1,2,\dots,6,\bullet = 0,1,\infty)$ are linearly independent in $\Xi_t$ for very general $t\in T$.
The key of the proof is to compute $\Dc(\nu_{\tr}((\rho^i)^*\xi_{\bullet}))$ for each $i=1,2,\dots,6$ and $\bullet = 0,1,\infty$.

Before proceeding to the proof, we explain geometric constructions of cycles in $\Xi_t$.
For $\rho = (\rho_1,\rho_2,\tau,\zeta)\in G_{\X}$, let $\D_{\rho}$ be the $(1,1)$-curve on $\P^1\times \P^1$ which passes through $(\rho_1^{-1}(0),\rho_2^{-1}(0))$, $(\rho_1^{-1}(1),\rho_2^{-1}(1))$ and $(\rho_1^{-1}(\infty),\rho_2^{-1}(\infty)) \in \Sigma^2$.
The local equation of $\D_{\rho}$ is given by $\rho_1^\sharp(x) = \rho_2^\sharp(y)$. Let $\Cc_{\rho}\subset \X_t$ be its pull-back by $\X_t\rightarrow \Y_t =\P^1\times \P^1$ and $\widetilde{\Cc}_{\rho}\subset \widetilde{\X}_t$ be the strict transform of $\Cc_{\rho}$.
By definition, $\D_{\rho}$ coincides with the inverse image of $\D$ by $\rho$,
so we have $\widetilde{\Cc}_{\rho} =\rho^{-1}_t(\widetilde{\Cc})$.
Moreover, by considering the $G_{\X}$-action on $\Sigma^2$, we have $Q_{(\rho_1^{-1}(\bullet), \rho_2^{-1}(\bullet))}=\rho_t^{-1}(Q_{(\bullet,\bullet)})$ for $\bullet = 0,1,\infty$.
Thus we see that $(\rho^*\xi_\bullet)_t=\rho_t^*((\xi_{\bullet})_{\tau(t)})\:\:(\bullet \in \{0,1,\infty\})$ is represented by the formal sum
\begin{equation}\label{geompresentation}
(\rho^*\xi_\bullet)_t = \left(\widetilde{\Cc}_{\rho}, \psi_\bullet\circ \rho_t\right) + \left(Q_{(\rho_1^{-1}(\bullet), \rho_2^{-1}(\bullet))}, \varphi_\bullet \circ \rho_t\right).
\end{equation}

For example, for $\rho^1,\rho^2,\cdots,\rho^6\in G_{\X}$ defined above, $\D_{\rho^1},\D_{\rho^2},\dots,\D_{\rho^6}$ are graphs of rational functions $z,1-z,z/(1-z),1/(1-z),1/z$ and $(z-1)/z$ on $\P^1$, respectively.
Hence $((\rho^i)^*\xi_{\bullet})_t$ are higher Chow cycles made from such graphs.

\begin{rem}\label{geomrem}
The idea of the construction of higher Chow cycles from the curve $\D_{\rho}$ is of Terasoma. The referee pointed out that the explicit 18 cycles made from the graphs of six rational functions on $\P^1$ are enough for the rank estimate.
\end{rem}

\subsection{Proof of the main theorem}
We define subgroups $N^\can$ and $N$ of $\Gamma(T,\Qc)$ as follows:
\begin{equation*}
\begin{aligned}
&N^\can = \langle \nu_{\mathrm{tr}}(\xi_0),\nu_\tr(\xi_1),\nu_{\tr}(\xi_\infty)\rangle_\Z \\
&N = \langle \nu_{\tr}(\rho^*\xi_\bullet): \rho\in G_{\X}, \bullet \in \{0,1,\infty\}\rangle_\Z.
\end{aligned}
\end{equation*}
By Proposition \ref{transregprop}, the transcendental regulator map induces a surjective map $\Xi_t\twoheadrightarrow \ev_t(N)$ for each $t\in T$.
Hence to prove Theorem \ref{mainthm}, we have to examine the rank of $N$.
To get a rank estimate for $N$, we use the following proposition, whose proof will be given in the next section.
\begin{prop}\label{thxprop}
For $\rho=(\rho_1,\rho_2,\tau,\zeta)\in G_{\X}$, let $\Theta_{\rho}: \tau^*(\O_T^\an)^{\oplus 2}\rightarrow (\O_T^\an)^{\oplus 2}$ be the morphism defined by
\begin{equation*}
\begin{tikzcd}
\begin{pmatrix}
\varphi_1 \\
\varphi_2
\end{pmatrix}
\arrow[r,mapsto]
&\begin{pmatrix}
\chi(\rho)^{-1}\cdot \phi_1(\tau)^{-2}\cdot \tau^\sharp(\varphi_1)\\
\chi(\rho)^{-1}\cdot \phi_2(\tau)^{-2}\cdot \tau^\sharp(\varphi_2)
\end{pmatrix}.
\end{tikzcd}
\end{equation*}
where $\tau^\sharp:\tau^*\O_T^\an\rightarrow \O_T^\an; \varphi\mapsto \varphi\circ\tau$. Let $\xi$ be an algebraic family of higher Chow cycles on $\widetilde{\X}\rightarrow T$.
Then we have
\begin{equation*}
\Dc(\nu_{\mathrm{tr}}(\rho^*\xi)) = \Theta_{\rho}(\Dc(\nu_{\mathrm{tr}}(\xi)))
\end{equation*}
where $\Dc:\Qc\simeq \Qc_\omega \rightarrow (\O_T^{\an})^{\oplus 2}$ is the Picard-Fuchs differential operator defined in Proposition \ref{preprop1}.
\end{prop}

Thanks to this proposition, we can compute $\Dc(N)$ explicitly and we get the desired rank estimate for $N$.
First, we compute $\Dc(N^\can)$.

\begin{prop}\label{ncanprop}
The images of $\nu_\tr(\xi_0),\nu_\tr(\xi_1),\nu_\tr(\xi_\infty)$ under the Picard-Fuchs differential operator $\Dc$ are as follows:
\begin{equation*}
\Dc(\nu_\tr(\xi_0)) = \frac{2}{a-b}
\begin{pmatrix}
1 \\[1.5ex]
-1
\end{pmatrix}
,\:\:
\Dc(\nu_\tr(\xi_1))= \frac{2}{a-b}
\begin{pmatrix}
\frac{\sqrt{1-b}}{\sqrt{1-a}}\\[1.5ex]
-\frac{\sqrt{1-a}}{\sqrt{1-b}}
\end{pmatrix}
,\:\:
\Dc(\nu_\tr(\xi_\infty)) = \frac{2}{a-b}
\begin{pmatrix}
\frac{\sqrt{b}}{\sqrt{a}} \\[1.5ex]
-\frac{\sqrt{a}}{\sqrt{b}}
\end{pmatrix}.
\end{equation*}
In particular, $\rank N^\can = 3$.
\end{prop}

\begin{proof}
The key for the proof is to find $\rho^a,\rho^b\in G_{\X}$ such that 
\begin{equation}\label{rhoabaction}
(\rho^a)^*(\xi_1-\xi_0) = \xi_1+\xi_0,\quad (\rho^b)^*(\xi_1-\xi_0) = \xi_0-\xi_\infty.
\end{equation}
We consider the following elements.
\begin{enumerate}
\renewcommand{\labelenumi}{(\alph{enumi})}
\item $\rho^a = (\id,\id,\tau^a,1)$ where $\tau^a\in G_T$ satisfies $(\tau^a)^\sharp\left(\sqrt{a}\right)=\sqrt{a}$, \\
$(\tau^a)^\sharp\left(\sqrt{1-a}\right)=-\sqrt{1-a}$, $(\tau^a)^\sharp\left(\sqrt{b}\right)=\sqrt{b}$ and $(\tau^a)^\sharp\left(\sqrt{1-b}\right)=\sqrt{1-b}$.\\
Then we have $\phi_1(\tau^a) = -1$ and $\phi_2(\tau^b)=1$.
\item $\rho^b = ((1\:\:\:\infty),(1\:\:\:\infty),\tau^b,1)$ where $\tau^b\in G_T$ satisfies $(\tau^b)^\sharp\left(\sqrt{a}\right)=\sqrt{1-a}$, \\
$(\tau^b)^\sharp\left(\sqrt{1-a}\right)=\sqrt{a}$, $(\tau^b)^\sharp\left(\sqrt{b}\right)=\sqrt{1-b}$ and $(\tau^b)^\sharp\left(\sqrt{1-b}\right)=\sqrt{b}$.\\
Then we have $\phi_1(\tau^b)=1$ and $\phi_2(\tau^b)=1$. 
\end{enumerate}
These elements stabilize the curve $\widetilde{\Cc}$. By the local description of the $G_{\X}$-action in the end of Section 6, we see that these elements satisfy (\ref{rhoabaction}).

By Proposition \ref{thxprop} and (\ref{rhoabaction}), we can compute $\Dc(\nu_\tr(\xi_0+\xi_1))$ and $\Dc(\nu_\tr(\xi_0-\xi_\infty))$ from Proposition \ref{preprop5} as follows:
\begin{equation*}
\begin{aligned}
&\Dc(\nu_\tr(\xi_0+\xi_1)) =\Theta_{\rho^a}\left(\Dc(\nu_{tr}(\xi_1-\xi_0))\right) = \frac{2}{a-b}
\begin{pmatrix}
1+ \frac{\sqrt{1-b}}{\sqrt{1-a}} \\[1.5ex]
-1-\frac{\sqrt{1-a}}{\sqrt{1-b}}
\end{pmatrix}
\\
&\Dc(\nu_\tr(\xi_0-\xi_\infty))= \Theta_{\rho^b}\left(\Dc(\nu_{tr}(\xi_1-\xi_0))\right)  = \frac{2}{(1-a)-(1-b)}
\begin{pmatrix}
\frac{\sqrt{b}}{\sqrt{a}}-1 \\[1.5ex]
1-\frac{\sqrt{a}}{\sqrt{b}}
\end{pmatrix}.
\end{aligned}
\end{equation*}
Hence the images of $\nu_\tr(\xi_0),\nu_\tr(\xi_1)$ and $\nu_\tr(\xi_\infty)$ under the Picard-Fuchs differential operator are as in the statement. Since they are linearly independent over $\Q$, the latter part follows from these expressions.
\end{proof}

Next, we give an estimate for the rank of $N$. 
The proof below was simplified by advice from Terasoma. Furthermore, Sreekantan pointed out an error in the Table \ref{basisofimage} in the previous version of this paper.
\begin{prop}\label{rk18}
$\rank N\ge 18$.
\end{prop}
\begin{proof}
Let $\rho^1,\rho^2,\dots,\rho^6 \in G_{\X}$ be elements defined after Theorem \ref{mainthm}.
Using Proposition \ref{thxprop} and Proposition \ref{ncanprop}, we can compute $\Dc(\nu_\tr((\rho^i)^*\xi_\bullet))$ for $i = 1,\dots, 6$ and $\bullet = 0,1,\infty$ as Table \ref{basisofimage}. 
Note that this table contains the ambiguity of signs because we do not fix the $G_{T}$-components of $\rho^i$.

{\small 
\begin{table}[h]
\caption{The images $\nu_\tr((\rho^i)^*\xi_\bullet)$ under the Picard-Fuch differential operator $\Dc$}
\label{basisofimage}
{\renewcommand\arraystretch{1.5}
\begin{tabular}{c|c}
its image in $G_{\Y_0}$ & $\Dc(\nu_\tr((\rho^i)^*\xi_0)), \quad \Dc(\nu_\tr((\rho^i)^*\xi_1)), \quad\Dc(\nu_\tr((\rho^i)^*\xi_\infty))$  \\
\hline 
&\\
$(\id,\id)$ & 
$\pm \dfrac{2}{a-b}
\begin{pmatrix}
1 \\[1.5ex]
-1
\end{pmatrix}
$
,
$\pm\dfrac{2}{a-b}
\begin{pmatrix}
\dfrac{\sqrt{1-b}}{\sqrt{1-a}} \\[1.5ex]
-\dfrac{\sqrt{1-a}}{\sqrt{1-b}}
\end{pmatrix}$
,
$\pm\dfrac{2}{a-b}
\begin{pmatrix}
\dfrac{\sqrt{b}}{\sqrt{a}}  \\[1.5ex]
-\dfrac{\sqrt{a}}{\sqrt{b}} 
\end{pmatrix}$
\\
&
\\
$(\id,(0\:1))$ &
$\pm \dfrac{2}{ab-a-b}
\begin{pmatrix}
\sqrt{1-b} \\[1.5ex]
\dfrac{1}{\sqrt{1-b}}
\end{pmatrix}
$,
$\pm \dfrac{2}{ab-a-b}
\begin{pmatrix}
\dfrac{1}{\sqrt{1-a}} \\[1.5ex]
\sqrt{1-a}
\end{pmatrix}$
,
$\pm \dfrac{2\sqrt{-1}}{ab-a-b}
\begin{pmatrix}
\dfrac{\sqrt{b}}{\sqrt{a}}  \\[1.5ex]
-\dfrac{\sqrt{a}}{\sqrt{b}} 
\end{pmatrix}$
\\
&
\\
$(\id,(1\:\infty))$ &
$\pm \dfrac{2\sqrt{-1}}{a+b-1}
\begin{pmatrix}
1 \\[1.5ex]
-1
\end{pmatrix}
$,
$\pm \dfrac{2\sqrt{-1}}{a+b-1}
\begin{pmatrix}
\dfrac{\sqrt{b}}{\sqrt{1-a}} \\[1.5ex]
-\dfrac{\sqrt{1-a}}{\sqrt{b}}
\end{pmatrix}$,
$\pm \dfrac{2\sqrt{-1}}{a+b-1}
\begin{pmatrix}
\dfrac{\sqrt{1-b}}{\sqrt{a}}  \\[1.5ex]
-\dfrac{\sqrt{a}}{\sqrt{1-b}} 
\end{pmatrix}$
\\
&
\\

$(\id,(0\:1\:\infty))$ &
$\pm \dfrac{2\sqrt{-1}}{a-ab-1}
\begin{pmatrix}
\sqrt{1-b} \\[1.5ex]
\dfrac{1}{\sqrt{1-b}}
\end{pmatrix}
$,
$\pm\dfrac{2}{a-ab-1}
\begin{pmatrix}
\dfrac{\sqrt{b}}{\sqrt{1-a}}  \\[1.5ex]
-\dfrac{\sqrt{1-a}}{\sqrt{b}} 
\end{pmatrix}$
,
$\pm\dfrac{2\sqrt{-1}}{a-ab-1}
\begin{pmatrix}
\dfrac{1}{\sqrt{a}} \\[1.5ex]
\sqrt{a}
\end{pmatrix}$
\\
&
\\
$(\id,(0\:\infty))$ &
$\pm \dfrac{2}{ab-1}
\begin{pmatrix}
\sqrt{b} \\[1.5ex]
\dfrac{1}{\sqrt{b}}
\end{pmatrix}
$
,
$\pm \dfrac{2\sqrt{-1}}{ab-1}
\begin{pmatrix}
\dfrac{\sqrt{1-b}}{\sqrt{1-a}}  \\[1.5ex]
-\dfrac{\sqrt{1-a}}{\sqrt{1-b}} 
\end{pmatrix}$
,
$\pm \dfrac{2}{ab-1}
\begin{pmatrix}
\dfrac{1}{\sqrt{a}} \\[1.5ex]
\sqrt{a}
\end{pmatrix}$
\\
&
\\
$(\id,(0\:\infty\:1))$ &
$\pm \dfrac{2\sqrt{-1}}{ab-b+1}
\begin{pmatrix}
\sqrt{b} \\[1.5ex]
\dfrac{1}{\sqrt{b}}
\end{pmatrix}
$,
$\pm \dfrac{2\sqrt{-1}}{ab-b+1}
\begin{pmatrix}
\dfrac{1}{\sqrt{1-a}} \\[1.5ex]
\sqrt{1-a}
\end{pmatrix}$,
$\pm \dfrac{2}{ab-b+1}
\begin{pmatrix}
\dfrac{\sqrt{1-b}}{\sqrt{a}}  \\[1.5ex]
-\dfrac{\sqrt{a}}{\sqrt{1-b}} 
\end{pmatrix}$

\\
\end{tabular}
}
\end{table}
}

We will show that the vectors in Table \ref{basisofimage} are linearly independent over $\Q$. It is enough to show that the first components of these vectors are linearly independent over $\C$. Note that the first components of these vectors are expressed by 
\begin{equation*}
2\zeta \cdot F_1\cdot F_2
\end{equation*}
where $\zeta\in\mu_4$, $F_1$ is either
\begin{equation}\label{F1}
\dfrac{1}{a-b},\:\:\dfrac{1}{a+b-1},\:\:\dfrac{1}{ab-a-b},\:\:\dfrac{1}{ab-b+1},\:\:\dfrac{1}{ab-1}\text{ or }\dfrac{1}{a-ab-1} \in \Frac(B_0)
\end{equation}
and $F_2$ is either
\begin{equation}\label{F2}
1,\:\dfrac{\sqrt{b}}{\sqrt{a}},\:\dfrac{\sqrt{1-b}}{\sqrt{1-a}},\:\dfrac{\sqrt{1-b}}{\sqrt{a}},\:\dfrac{\sqrt{b}}{\sqrt{1-a}},\:\dfrac{1}{\sqrt{1-a}},\: \sqrt{1-b},\:\dfrac{1}{\sqrt{a}}\text{ or }\sqrt{b} \in \Frac(B).
\end{equation}
Since elements in (\ref{F1}) are linearly independent over $\C$ and elements in (\ref{F2}) are linearly independent over $\Frac(B_0)$, their products are linearly independent over $\C$. Hence we have the result.
\end{proof}
Then we can prove the main theorem as follows.
\begin{proof}[(Proof of Theorem \ref{mainthm})] 
Since $N$ is countable, by Lemma \ref{basiclem}, $N\twoheadrightarrow \ev_t(N)$ is bijective for very general $t\in T$. By Proposition \ref{rk18}, $\rank \ev_t(N)\ge 18$ for such $t$. 
Since $\ev_t(\nu_{\tr}(\rho^*\xi_\bullet))=r((\rho^*\xi_\bullet)_t)$, $\ev_t(N)$ is generated by the images of $(\rho^*\xi_\bullet)_t\:(\rho\in G_{\X},\bullet \in\{ 0,1,\infty\})$ under the transcendental regulator map.
By Proposition \ref{transregprop}, the transcendental regulator map induces the surjective map $\Xi_t\twoheadrightarrow \ev_t(N)$.
Hence we have $\rank \Xi_t \ge \rank \ev_t(N)\ge 18$ for very general $t\in T$.
\end{proof}
\subsection{Precise rank of $N$}
\begin{prop}\label{rkex18}
$\rank N = 18$.
\end{prop}
To prove this, we use the following proposition which we prove in Section 8.
\begin{prop}\label{thxprop2}
\begin{enumerate}
\item For $\rho=(\rho_1,\rho_2,\tau,\zeta)\in G_{\X}$, let $\Psi_{\rho}:\tau^*\O_T^\an \rightarrow \O_T^\an$ be the $\O_T^\an$-module isomorphism defined by 
\begin{equation*}
\begin{tikzcd}
\varphi \arrow[r,mapsto]  & \chi(\rho)^{-1}\cdot \tau^\sharp(\varphi).
\end{tikzcd}
\end{equation*}
Then $(\Psi_{\rho})_{\rho\in G_{\X}}$ defines a $G_{\X}$-linearization on $\O_T^\an$. 
Furthermore, $(\Psi_{\rho})_{\rho\in G_{\X}}$ induces a $G_{\X}$-linearization on $\Qc_\omega = \O_T^\an/\Pc_\omega$. 
Under the isomorphism $\Qc_\omega\simeq \Qc$, $(\Psi_{\rho})_{\rho\in G_{\X}}$ defines a $G_{\X}$-linearization on $\Qc$.
\item Let $\xi$ be an algebraic family of higher Chow cycles on $\widetilde{\X}\rightarrow T$. For any $\rho\in G_{\X}$, we have 
\begin{equation*}
\nu_{\tr}(\rho^*\xi) = \Psi_{\rho}(\nu_{\tr}(\xi)).
\end{equation*}
\end{enumerate}
\end{prop}
The following proof of Proposition \ref{rkex18} was simplified by the advice from T. Saito.
\begin{proof}[(Proof of Proposition \ref{rkex18})] It is enough to prove $\rank N \le 18$.
By the $G_{\X}$-linearization $(\Psi_{\rho})_{\rho\in G_{\X}}$ of $\Qc$, $N$ becomes a right $G_{\X}$-module.

We consider the following two subgroups $H$ and $I$ of $G_{\X}$. 
We define the normal subgroup $H$ by $\Ker(G_{\X}\rightarrow G_{T})$.
By the definition of $\Psi_{\rho}$, $H$ acts on $N$ as $\pm 1$.
Next, we define the subgroup $I$ of $G_{\X}$ by 
\begin{equation*}
I =\{(\rho_1,\rho_2,\tau,\zeta)\in G_{\X}: \rho_1=\rho_2\in \mathfrak{S}(\{0,1,\infty\})\text{ and }\zeta=1\}
\end{equation*}
where we regard $\mathfrak{S}(\{0,1,\infty\})$ as a subgroup of $\mathfrak{S}(\Sigma)$ by $\{0,1,\infty\}\subset \Sigma$.
By definition, $I$ stabilizes the subset $\{(0,0),(1,1),(\infty,\infty)\}$ of $\Sigma^2$ on $\Y$.
Thus $\widetilde{\Cc}$ is stable under the $I$-action.
Hence for $\bullet = 0,1,\infty$ and $\rho\in I$, $(\rho^*\xi_\bullet)_t$ coincides with either $\pm (\xi_0)_t,\pm (\xi_1)_t$ or $\pm (\xi_\infty)_t$ up to decomposable cycles (see (\ref{geompresentation})).
By Proposition \ref{transregprop}, $\nu_\tr(\xi)$ is determined by $\xi_\ind$, therefore $\nu_\tr(\rho^*\xi_\bullet)$ coincides with either $\pm \nu_\tr(\xi_0),\pm \nu_\tr(\xi_1)$ or $\pm\nu_{\tr}(\xi_\infty)$.
This implies that $N^\can$ is an $I$-submodule of $N$.

Therefore $N^\can$ is a $H I$-submodule of $N$.
Then $N^\can \hookrightarrow N$ induces the following $G_{\X}$-module homomorphism.
\begin{equation}\label{indrepn}
\mathrm{Ind}_{H I}^{G_{\X}} N^\can \lra N
\end{equation}
where $\mathrm{Ind}_{H I }^{G_{\X}} N^\can$ is the induced representation.
Since the $G_{\X}$-module $N$ is generated by $N^\can$, the map (\ref{indrepn}) is surjective.
Then we have
\begin{equation*}
\rank N \le \rank \mathrm{Ind}_{H I}^{G_{\X}} N^\can = |G_{\X}/H I|\cdot \rank N^\can = 3\cdot  |G_{\X}/H I|.
\end{equation*}
Hence it is enough to show $ |G_{\X}/H I| = 6$. 
First, we show $H \cap I = \{\id\}$. 
Let $\rho=(\rho_1,\rho_2,\tau,\zeta)\in H\cap I$.
Then we have $\rho_1=\rho_2\in \mathfrak{S}(\{0,1,\infty\})$, $\tau = \id$ and $\zeta = 1$.
Furthermore, the image of $(\rho_1,\rho_2)\in G_{\Y_0}$ under $G_{\Y_0}\rightarrow G_{T_0}$ is $(\id,\id)$ by the condition of the fiber product.
Since $\mathfrak{S}(\{0,1,\infty\})\hookrightarrow \mathfrak{S}(\Sigma)\twoheadrightarrow \Aut(S_0)$ is the isomorphism, we see that $(\rho_1,\rho_2)=(\id,\id)$. 
Hence $\rho=\id$. 
Second, we calculate $|G_{\X}|$ and $|H|$.
We have that $|G_{\X}| = 2\cdot |G_{\Y}| = 2\cdot |\mathfrak{S}_4\times_{\mathfrak{S}_3}\mathfrak{S}_4|^2 = 2^{11}\cdot 3^2$.
By $|G_{T}| = |\mathfrak{S}_4^2| = 2^6\cdot 3^2$ and the fact that $G_{\X}\rightarrow G_{T}$ is surjective, $|H| =|G_{\X}|/|G_T| =2^5$.
Third, we calculate $|I|$. For each $\rho\in \mathfrak{S}(\{0,1,\infty\})$, the number of elements in $G_{\X}$ of the form $(\rho,\rho,*,1)$ equals to the cardinality of $\Ker(G_{T}\rightarrow G_{T_0})$, i.e. $2^4$.
Hence $|I|=2^4\cdot |\mathfrak{S}(\{0,1,\infty\})| = 2^5\cdot 3$.
To sum up these results, we have
\begin{equation*}
|G_{\X}/H I| = \frac{|G_{\X}|}{|H I|} = \frac{|G_{\X}|}{|H|\cdot |I|} = 6.
\end{equation*}
\end{proof}
\section{The group action on the Picard-Fuchs differential operator}
In this section, we prove Proposition \ref{thxprop} and Proposition \ref{thxprop2} and complete the proof of the main theorem. 
They are proved by examining the $G_{\X}$-actions on the periods and the Picard-Fuchs differential operator.
\subsection{The group action on periods}
First, we prove Proposition \ref{thxprop2}.
By the property of 1-cocycles, we see that $(\Psi_{\rho})_{\rho\in G_{\X}}$ defines a $G_{\X}$-linearization on $\O_T^\an$. 
To prove that this induces the $G_{\X}$-linearization on $\Qc_\omega$, we should prove that $\Psi_{\rho}$ preserves the subsheaf $\Pc_\omega$ which is the local system consisting of period function with respect to $\omega$.

Let $\rho=(\rho_1,\rho_2,\tau,\zeta)\in G_{\X}$. 
By the local description of the $G_{\X}$-action on $\widetilde{\X}$, we can check that $\rho^*\omega = \chi(\rho)\cdot \omega$.
Pulling-back this relation at $\tau(t)\in T$, we have 
\begin{equation}\label{rhotransformula}
(\rho_t^{-1})^*(\omega_{t}) = \chi(\rho^{-1})(\tau(t))\cdot \omega_{\tau(t)}\quad  \in \Gamma(\widetilde{\X}_{\tau(t)},\Omega^2_{\widetilde{\X}_\tau(t)/\C}).
\end{equation}
Here $\rho_t:\widetilde{\X}_t\rightarrow \widetilde{\X}_{\tau(t)}$ is the isomorphism in the diagram (\ref{rhot}) and $\chi(\rho^{-1})(\tau(t))\in \C$ is the value of $\chi(\rho^{-1})\in B$ at $\tau(t)$.
From the equation (\ref{rhotransformula}), for any 2-chain $\Gamma_{\tau(t)}$ on $\widetilde{\X}_{\tau(t)}$, we have 
\begin{equation}\label{sekibunprop}
\begin{aligned}
\int_{\rho_t^{-1}(\Gamma_{\tau(t)})}\omega_t &= \int_{\Gamma_{\tau(t)}}(\rho^{-1}_t)^*(\omega_t) = \chi(\rho^{-1})(\tau(t))\cdot \int_{\Gamma_{\tau(t)}}\omega_{\tau(t)}  \\
&= \tau^\sharp\left(\chi(\rho^{-1})\right)(t)\cdot \int_{\Gamma_{\tau(t)}}\omega_{\tau(t)} = \chi(\rho)(t)^{-1}\cdot \int_{\Gamma_{\tau(t)}}\omega_{\tau(t)}
\end{aligned}
\end{equation}
where the last equality follows from the property of 1-cocycles. This is the key for the proof of Proposition \ref{thxprop2}.
\begin{proof}[(Proof of Proposition \ref{thxprop2})]
We prove (1).
We prove that $\Psi_{\rho}(\tau^*\Pc_\omega) = \Pc_{\omega}$. 
Since $\rho^*(\Psi_{\rho^{-1}})=\Psi_{\rho}^{-1}$, it is enough to show only $(\subset)$.
Let $U'$ be an open subset of $T$ in the classical topology and $\varphi\in \tau^*\Pc_\omega(U')=\Pc_\omega(\tau(U'))$.
Then there exists a $C^\infty$-family of 2-cycles $\{\Gamma_t\}_{t\in \tau(U)}$ such that $\varphi(t) = \int_{\Gamma_t}\omega_t$.
For any $t\in U'$, we have 
\begin{equation*}
(\Psi_{\rho}(\varphi))(t) = \chi(\rho)(t)^{-1}\cdot \varphi(\tau(t)) = \chi(\rho)(t)^{-1}\cdot \int_{\Gamma_{\tau(t)}}\omega_{\tau(t)} \underset{(\ref{sekibunprop})}{=} \int_{\rho_t^{-1}(\Gamma_{\tau(t)})}\omega_t\in \Pc(\omega_t).
\end{equation*}
Hence we have $\Psi_{\rho}(\varphi)\in \Pc_\omega(U')$.
Thus $(\Psi_{\rho})_{\rho\in G_{\X}}$ induces the $G_{\X}$-linearization on $\Qc_\omega = \O_T^\an/\Pc_\omega$.

Next, we prove (2).
By Corollary \ref{basiccor}, it is enough to show that 
\begin{equation}\label{ETS}
\langle \ev_t\left(\nu_{\tr}(\rho^*\xi)\right), [\omega_t]\rangle = \langle \ev_t\left(\Psi_{\rho}(\nu_{\tr}(\xi))\right), [\omega_t]\rangle
\end{equation}
for any $t\in T$.
We take a neighborhood $U'$ of $\tau(t)$ and a $C^\infty$-family of 2-chains $\{\Gamma_{t'}\}_{t'\in U'}$ such that $\Gamma_{t'}$ is a 2-chain associated with $\xi_{t'}$.
We will compute the right-hand side of (\ref{ETS}).
Under the isomorphism $\Qc\xrightarrow{\sim}\Qc_\omega$,  $\Psi_{\rho}(\nu_\tr(\xi))$ is represented by the local holomorphic function 
\begin{equation*}
\tau^{-1}(U')\ni t' \longmapsto \chi(\rho)(t')^{-1}\int_{\Gamma_{\tau(t')}}\omega_{\tau(t')} \underset{(\ref{sekibunprop})}{=} \int_{\rho_{t'}^{-1}\left(\Gamma_{\tau(t')}\right)}\omega_{t'}\in \C.
\end{equation*}
The right-hand side of (\ref{ETS}) is represented by the value of this function at $t$. i.e.
\begin{equation}\label{ans}
\langle \ev_t\left(\Psi_{\rho}(\nu_{\tr}(\xi))\right), [\omega_t]\rangle=\displaystyle \int_{\rho_t^{-1}(\Gamma_{\tau(t)})}\omega_{t} \mod \Pc(\omega_{t}).
\end{equation}
On the other hand, the left-hand side of (\ref{ETS}) is $\langle r((\rho^*\xi)_t), [\omega_t]\rangle  = \langle r(\rho_t^*(\xi_{\tau(t)})), [\omega_t]\rangle $.
Since $\rho_t^{-1}(\Gamma_{\tau(t)})$ is a 2-chain associated with $\rho_t^*(\xi_{\tau(t)})$, the left-hand side of (\ref{ETS}) also coincides with (\ref{ans}).
Hence we have the result.
\end{proof}
\subsection{The group action on the Picard-Fuchs operator}
Next, we will investigate the $G_{\X}$-action on the differential operator $\Dc$.
For $\tau \in G_{T}$, we define differential operators $\Dc^{\tau}_1,\Dc^{\tau}_2:\O_T^\an \rightarrow \O_T^\an$ as follows:
\begin{equation*}
\begin{aligned}
&\Dc^{\tau}_1 = a'(1-a') \frac{\del^2}{(\del a')^2} +(1-2a')\frac{\del}{\del a'} -\frac{1}{4}  \\
&\Dc^{\tau}_2 = b'(1-b') \frac{\del^2}{(\del b')^2} +(1-2b')\frac{\del}{\del b'} -\frac{1}{4}
\end{aligned}
\end{equation*}
where $a'=\tau^\sharp(a)$ and $b'=\tau^\sharp(b)$.
By definition, 
\begin{equation}\label{twistD}
\Dc^\tau_i(\tau^\sharp(\varphi)) = \tau^\sharp(\Dc_i(\varphi))\:\:(i=1,2)
\end{equation}
holds for any local section $\varphi$ of $\O_T^\an$.
We have the following transformation formulas.
\begin{prop}\label{transformationformulae}
For $\rho=(\rho_1,\rho_2,\tau,\zeta)\in G_{\X}$, we have the following relations in the ring of differential operators on $\O_T^\an$.
\begin{equation*}
\begin{aligned}
&\Dc_1\cdot \chi(\rho)^{-1} = \chi(\rho)^{-1}\cdot \phi_1(\tau)^{-2}\cdot \Dc^\tau_1 \\
&\Dc_2\cdot \chi(\rho)^{-1} = \chi(\rho)^{-1}\cdot \phi_2(\tau)^{-2}\cdot \Dc^\tau_2 
\end{aligned}
\end{equation*}
where we regard $\chi(\rho),\phi_1(\tau)$ and $\phi_2(\tau)$ as differential operators by multiplication.
\end{prop}
\begin{proof}
We prove the first equation.
The second equation is proved similarly. Since $\chi(\rho) = \zeta\phi_1(\tau)\phi_2(\tau)$ and $\zeta\phi_2(\tau)$ commutes with $\Dc_1$, it is enough to show that 
\begin{equation}\label{ETS2}
\phi_1(\tau)^3\cdot \Dc_1\cdot \phi_1(\tau)^{-1} = \Dc_1^\tau.
\end{equation}
Let $\tau = (\tau_1,\tau_2)\in G_{T}=\mathfrak{S}_4^2$.
We denote the image of $\tau_1$ under $\mathfrak{S}_4\rightarrow \Aut(S_0)=\mathfrak{S}(\{0,1,\infty\})$ by $\underline{\tau}_1$.
By (\ref{sqrteqn}) and the table in Proposition \ref{chikeisan}, $\phi_1(\tau)$ is determined by $\underline{\tau}_1$ up to sign and coincides with either $\pm 1, \pm \sqrt{-1}\sqrt{1-a}$ or $\pm \sqrt{-1}\sqrt{a}$.
Using $\frac{\del}{\del a}\cdot a^\lambda = \lambda a^{\lambda-1}+a^{\lambda}\frac{\del}{\del a}$ and $\frac{\del}{\del a}\cdot (1-a)^\lambda = -\lambda (1-a)^{\lambda-1}+(1-a)^{\lambda}\frac{\del}{\del a}$, we can compute the left-hand side of (\ref{ETS2}) as follows:
\begin{equation}\label{LHS}
\begin{aligned}
&a(1-a)\dfrac{\del^2}{\del a^2} + (1-2a)\dfrac{\del}{\del a} -\dfrac{1}{4} && (\underline{\tau}_1 = \id, (0\:\:1)) \\
&-a(1-a)^2\frac{\del^2}{\del a^2}-(1-a)^2\frac{\del}{\del a} -\frac{1}{4} && (\underline{\tau}_1 = (1\:\;\infty), (0\:\:1\:\:\infty)) \\
&-a^2(1-a)\frac{\del^2}{\del a^2}+a^2\frac{\del}{\del a}-\frac{1}{4}  &&  (\underline{\tau}_1 = (0\:\:\infty), (0\:\:\infty\:\;1))
\end{aligned}
\end{equation}
On the other hand, $\Dc_1^\tau$ is also determined by $\underline{\tau}_1\in \Aut(S_0)=\mathfrak{S}(\{0,1,\infty\})$ by definition.
We can check (\ref{ETS2}) holds for each $\underline{\tau}_1$.
For example, when $\underline{\tau}_1 = (1\:\:\infty)$, the differential operator $\Dc_1^\tau$ is computed as follows: since $a' = \frac{a}{a-1}$, we have 
\begin{equation*}
\begin{aligned}
&\frac{\del}{\del a'} = \frac{\del a}{\del a'}\cdot \frac{\del }{\del a} = -\frac{1}{(a'-1)^2}\cdot \frac{\del}{\del a} = -(a-1)^2\cdot \frac{\del }{\del a} \\
&\frac{\del^2}{(\del a')^2}= \left(-(a-1)^2\cdot \frac{\del}{\del a}\right)^2 = (a-1)^4\frac{\del^2}{\del a^2}+2(a-1)^3\frac{\del }{\del a}.
\end{aligned}
\end{equation*}
By substituting $a',\frac{\del}{\del a'}, \frac{\del^2}{(\del a')^2}$ in $\Dc_1^{\tau}$ by them, we can confirm that $\Dc_1^\tau$ coincides with the second differential operator in (\ref{LHS}).
\end{proof}

Finally, we prove Proposition \ref{thxprop}.

\begin{proof}[(Proof of Proposition \ref{thxprop})] 
By Proposition \ref{transformationformulae}, for a local section $\varphi$ of $\O_T^\an$, we have 
\begin{equation*}
\begin{aligned}
\Dc_1\left(\Psi_{\rho}(\varphi)\right) &= \Dc_1\left(\chi(\rho)^{-1}\cdot \tau^\sharp(\varphi)\right) = \chi(\rho)^{-1}\cdot \phi_1(\tau)^{-2}\cdot \Dc_1^\tau(\tau^\sharp(\varphi)) \\
&\underset{(\ref{twistD})}{=} \chi(\rho)^{-1}\cdot \phi_1(\tau)^{-2}\cdot \tau^\sharp(\Dc_1(\varphi)) \\
\Dc_2\left(\Psi_{\rho}(\varphi)\right) &=  \chi(\rho)^{-1}\cdot \phi_2(\tau)^{-2}\cdot \tau^\sharp(\Dc_2(\varphi))
\end{aligned}
\end{equation*}
In particular, we have $\Dc(\Psi_{\rho}(\varphi)) = \Theta_{\rho}(\Dc(\varphi))$. 
Then by Proposition \ref{thxprop2}, we have 
\begin{equation*}
\Dc(\nu_{\tr}(\rho^*\xi)) = \Dc(\Psi_{\rho}(\nu_\tr(\xi))) = \Theta_{\rho}(\Dc(\nu_{\tr}(\xi))).
\end{equation*}
\end{proof}

\bibliographystyle{plain}
\bibliography{reference}

\end{document}